\documentclass{article}

\usepackage[top=1.5in,left=1.25in,right=1.25in,bottom=1.5in]{geometry}

\usepackage{latexsym}
\usepackage{epsfig, ecltree, epic, eepic}
\usepackage{enumerate,amsmath,amssymb,dsfont,pifont,amsthm}
\usepackage[dvips]{color}
\usepackage{graphicx}
\usepackage[arrow, matrix, curve]{xy}

\numberwithin{equation}{section}

\usepackage{amsthm,amsmath}
\usepackage{enumerate,amssymb,dsfont,pifont}

\theoremstyle{plain}
\newtheorem{theorem}{Theorem}[section]
\newtheorem{corollary}[theorem]{Corollary}
\newtheorem{lemma}[theorem]{Lemma}
\newtheorem{proposition}[theorem]{Proposition}

\theoremstyle{definition}
\newtheorem{definition}[theorem]{Definition}

\theoremstyle{remark}
\newtheorem{remark}[theorem]{Remark}

\newtheorem{example}[theorem]{Example}

\newcommand{\bbc}{\mathbb{C}}
\newcommand{\bbr}{\mathbb{R}}

\newcommand{\bbp}{\mathbb{P}}
\newcommand{\bbe}{\mathbb{E}}

\newcommand{\cb}{\mathcal{B}}

\newcommand{\norm}[1]{\left\| #1 \right\|}

\renewcommand{\leq}{\leqslant}
\renewcommand{\geq}{\geqslant}

\newcommand{\NN}{\mathbb{N}}

\newcommand{\RR}{\mathbb{R}}

\newcommand{\SD}{\mathbb{S}^{d-1}}

\newcommand{\EE}{\mathbb{E}}
\newcommand{\PP}{\mathbb{P}}

\newcommand{\calA}{\mathcal{A}}
\newcommand{\calB}{\mathcal{B}}

\newcommand{\calF}{\mathcal{F}}

\theoremstyle{plain}

\theoremstyle{definition}

\newcommand{\change}[1]{\textcolor{black}{#1}}

\begin{document}

\allowdisplaybreaks

\title{\bfseries OPERATOR-STABLE-LIKE PROCESSES}

\author{\textsc{Peter Scheffler, Alexander Schnurr,  Daniel Schulte} \thanks{{Department of Mathematics, University of Siegen, 57068 Siegen, Germany}
{corresponding author: schnurr@mathematik.uni-siegen.de}}
    }

\date{\today}

\maketitle

\begin{abstract}
In the present paper, we introduce so-called operator-stable-like processes. Roughly speaking, they behave locally like operator-stable processes, but they need not to be homogenous in space. 
Having shown existence for this class of processes,
we analyze maximal estimates, the existence of moments, the short- and long-time behavior of the sample paths and $p$-variation.
The class introduced here includes stable-like processes as special case. 
\end{abstract}



\section{Introduction}

Let $X=(X_t)_{t\geq 0}$ be a \emph{L\'evy process}, that is, a process with c\`adl\`ag paths in $\bbr^d$ having stationary and independent increments (cf. \cite{Sato}). It is a well-known fact that the characteristic functions of the L\'evy process can be written as
\[
\bbe^x\left(e^{i\xi'(X_t-x)}\right) = \bbe^0 \left(e^{i\xi'X_t}\right) = e^{-t \psi(\xi)}
\]
where $x$ denotes the starting point of the space-homogeneous process.  The \emph{L\'evy exponent} $\psi : \RR^d  \rightarrow \mathbb{C}$ has the following representation:
\begin{align*}
\begin{split}
\psi(\xi)= & -i\langle l , \xi \rangle +\frac{1}{2} \langle \xi ,  Q \xi  \rangle + \int_{\Gamma} \left( 1-e^{i \langle y , \xi \rangle }+1_{\lbrace \| y \| < 1 \rbrace } i  \langle \xi , y \rangle \right) \, \phi(dy).
\end{split}
\end{align*}
Here, $l$ is a vector in $\bbr^d$, $Q$ is a positive semindefinite $d\times d$-matrix and $\phi$ is the so called \emph{L\'evy measure}; 
$\Gamma$ denotes $\RR^d \setminus \lbrace 0 \rbrace$. 
The tupel $(l,Q, \phi(\cdot))$ is called \emph{L\'evy triplet} of the process $X$. 

Among the various subclasses of L\'evy processes, so called \emph{operator-stable processes} are a prominent one (cf. \cite{Sharpe}, \cite{Peter}  and the references given therein).
Operator-stable L\'evy processes admit the triplet $(l,Q,\phi)$ where
\[
  \phi(A)=\int_{\SD} \int_0^\infty 1_A(r^E\theta) \frac{dr}{r^2} \ \sigma(d\theta), \hspace{10mm} A\in \calB(\RR^d\setminus\{0\}).
\]
Here, $E$ is a symmetric $d\times d$-matrix and $\sigma$ is a finite measure on the unit sphere $\SD$ (cf. \cite{Peter} Theorem 7.2.5.).
 
Operator-stable processes have been used in order to model multivariate financial and hydrological data (cf. \cite{Nolan} Section 9 and \cite{Water}). In both cases, the datasets appear to be inhomogeneous in space. Hence, model classes which allow for this kind of inhomogeneity are desirable. Our aim is thus to construct a class of processes which behaves locally like operator-stable L\'evy processes, but which is \emph{inhomogeneous} in space and belongs to a class of stochastic processes, which is still analytically tractable. 

To this end, we modify the L\'evy measure of an operator stable law without normal component in such a way that the exponent, that is, the $d \times d-$ matrix $E$, is no longer constant but depends on the position $x \in \RR^d$ in space. If the exponent $E(x)$ satisfies certain conditions, we call the resulting family of L\'evy measures $\phi(x, \cdot)$ operator-stable-like L\'evy measures.
Unlike in the L\'evy case, it is by no means clear, that a corresponding process exists. 
According to the L\'evy measure, we construct an associated stochastic differential equation (SDE) and prove that the solution of this SDE is an It\^o process in the sense of \cite{vierleute}. Hence it is both: a Markov process and a semimartingale with `nice' characteristics.  We show that for symmetric operator-stable-like L\'evy measures the symbol of the constructed process can be represented with these L\'evy measures. 

In the second part we investigate properties of the symbol of the new class of processes. We show a scaling property which is helpful to get lower and upper bounds for the symbol. Then we use these estimates to study the properties of the processes. We will focus on maximal estimates, the existence of moments, the short- and long-time behavior of the sample paths and the $p$-variation.

\vspace{2mm}
The notation we are using is more or less standard: 
We write $\mathbb{N}:=\lbrace 1,2,... \rbrace$  and $\mathbb{N}_0:= \mathbb{N} \cup \lbrace 0 \rbrace$. 
We use  $ \Gamma:=\RR^d \setminus \lbrace 0 \rbrace$.\\ 
For $f,g : \RR^d \rightarrow \RR$ we write $f \sim g$, if there are $C_1,C_2>0$ such that  $ C_1 f(x) \leq g(x) \leq C_2 f(x)$. 
Vectors in $\mathbb{R}^d$ are column vectors; for the transposed vector we write $^{t}$ and the components are denoted by $x^{(j)}$ $(j=1,...,d)$.  The Euclidean scalar product of $x,y$ is denoted by $\langle x, y \rangle$ and $\| \cdot \| := \| \cdot \| _{2}$ denotes the Euclidean norm. $| \cdot |$ denotes absolute value respectively the 1-norm on $\RR^d$ and $L(\RR^d)$ where the latter denotes the set of endomorphisms on $\RR^d$, written as matrices. The operator norm on $L(\RR^d)$ is written as
\begin{equation}
\|A \|:= \sup\limits_{\|x\| =1} \|Ax \| = \sup\limits_{\|x\| \leq 1} \|Ax \| .
\end{equation} 
The matrix exponential 
\begin{equation}
r^A:= \exp( A \ln r)=\sum\limits_{k=0}^{\infty} \frac{A^k}{k!}(\ln r)^k \quad \in L(\RR^d),
\end{equation}
plays a vital role in our considerations (for some details and estimates in this context cf. Appendix A.). For a state-space dependent matrix  $E(x) \in L(\RR^d)$ we define 
$$ \lambda(x): = \min \lbrace \mathrm{Re } \ \sigma(x) : \sigma(x) \text{ is eigenvalue of } E(x) \rbrace$$
and
$$ \Lambda(x): =\max \lbrace \mathrm{Re } \ \sigma(x) : \sigma(x) \text{ is eigenvalue of } E(x)  \rbrace.$$
The open Ball is denoted by $B_R(x):= \lbrace y \in \mathbb{R}^d : \| y-x\| < R \rbrace$.
We write $M^b(\mathbb{R}^d)$ for the set of bounded Borel measures on $\mathbb{R}^d$ and $M^1(\mathbb{R}^d)$ for probability measures.  We always consider a stochastic basis $(\Omega, \calF, (\calF_t)_{t \geq 0}, \PP)$ in the background which satisfies the usual hypotheses. 

$B(\mathbb{R}^d):=B(\mathbb{R}^d,\mathbb{R})$ are the Borel measurable functions $f : \mathbb{R}^d \rightarrow \mathbb{R}$ and $B_b(\mathbb{R}^d)$ the bounded Borel measurable functions. Continuous functions are denoted by $C(\mathbb{R}^d)$. Hence, $C_b(\mathbb{R}^d)$, $ C_{\infty}(\mathbb{R}^d)$ resp. $C_c(\mathbb{R}^d)$ are the bounded continuous funcitons, those vanishing at infinity resp. those with compact support. The upper index $C^i(\mathbb{R}^d)$ resp. $C^{\infty}(\mathbb{R}^d)$ denotes continuous differentiability. For the Schwartz space we write $S(\mathbb{R}^d)$. 
For $u \in S(\mathbb{R}^d)$
\begin{equation}
\hat{u}(\xi) := \frac{1}{(2 \pi)^d} \int_{\RR^d} e^{-i \langle x, \xi \rangle} u(x) dx
\end{equation} 
is the \emph{Fourier transform} and for $v \in S(\RR^d)$ the inverse transform is given by
\begin{equation}
 \check{v}(x)= \int_{\RR^d} e^{i \langle x, \xi \rangle} v(\xi) d \xi.
\end{equation}

\vspace{2mm}
The paper is organized as follows: 
In Section 2 we construct the new class of processes. Section 3 is devoted to properties of this class.
Some useful results on matrix exponentials and generalized polar coordinates are postponed to Appendix A.

\section{Construction of Operator-stable-like Processes}

The class we would like to consider should behave locally like operator-stable L\'evy processes, but the transition functions should be state-space dependent. In order to construct such a class, we have to leave the L\'evy framework behind. It turns out, that the class of It\^o processes in the sense of Cinlar et al. is sufficiently rich for our purpose and it is still analytically tractable \cite{vierleute}. 

For the readers convenience, we shortly recall that a process
\begin{equation}
  \left( \Omega, \calA, \PP^x,(X_t)_{t \geq 0}\right) _{x \in \RR^d} ,
\label{Formel:FamStPro}
\end{equation}
on $\bbr^d$ which might start in each point of the state space, 
is called  \emph{universal Markov process}, if 
\begin{enumerate} [leftmargin{0mm}]
\item[(MP1)] for each $ A \in \calA$ the mapping $x \mapsto \PP^x(A)$  is measurable;
\item[(MP2)] for each $x\in \RR^d$ we have $\PP^x(X_0=x)=1$;
\item[(MP3)] for all $s,t \geq 0, x \in \RR^d$ and $B \in \calB^d$ it holds that
$$\PP^x(X_{s+t}\in B | \calF_s^X)= \PP^{X_s}(X_t \in B)  \quad  \PP^x \text{-f.s. } $$
\end{enumerate}
Compare in this context the classical monograph \cite{blumenthalget}. 

\begin{definition}
An \emph{It\^o process}  is a stochastic process $(X,\bbp^x)_{x\in\bbr^d}$ which is universal Markov and a semimartingale for every $x\in\bbr^d$ with characteristics of the form 
\begin{align} \begin{split}
  B_t^{(j)}(\omega) &=\int_0^t  \ell^{(j)}(X_s(\omega)) \ ds,  \hspace{10mm} j=1,...,d\\
  C_t^{jk}(\omega)  &=\int_0^t Q^{jk}(X_s(\omega)) \ ds,       \hspace{10mm} j,k=1,...,d\\
  \nu(\omega;ds,dy) &=N(X_s(\omega),dy) \ ds 
\end{split} \end{align}
for every $x\in\bbr^d$ with respect to the fixed truncation function $y1_{\lbrace \| y \| <1 \rbrace}$. Here $\ell(x)=(\ell^{(1)}(x),...,\ell^{(d)}(x))'$ is a vector in $\bbr^d$, $Q(x)$ is a positive semi-definite matrix and $N$ is a Borel transition kernel such that $N(x,\{0\})=0$. We call $\ell$, $Q$ and $N$
the \emph{differential characteristics} of the process.
\end{definition}

Sometimes this class is called \emph{L\'evy-type processes}, or, if the Markov property is not demanded, \emph{homogeneous diffusions with jumps}. The differential characteristics of a L\'evy process are $(\ell, Q, \phi)$. Like the process, the differential characteristics are homogeneous in space in this simple case. In case of an operator-stable L\'evy process, they are $(0,0,\phi)$ where $\phi$ is operator stable. 

In \cite{cinlarjacod} the authors have proved the close relationship between It\^o processes and SDEs of Skorokhod type.  We will construct our class of processes via stochastic differential equations of this kind. For other methods of construction cf. Chapter 3  in \cite{LevyLike}.

There is a close relationship between operator-stable-like Le\'vy measures and L\'evy measures of operator-stable distributions (cf. Chapter 7 in \cite{Peter}). Roughly speaking, operator-stable-like Le\'vy measures are L\'evy measures of operator-stable distributions where the exponent $E \in GL(\RR^d)$ may vary in space. In order to give a rigorous definition we introduce the following linear operators.

\begin{definition}
We call \emph{admissible exponents} the class of linear operators $ E(x) \in GL(\RR^d)$, $ x \in \RR^d$, with the following properties: 
\begin{enumerate}
\item[(E1)] $E(x)$ is symmetric.
\item[(E2)] $ E(x)$ is Lipschitz, that is, there exists a constant $C>0$, s.t for every $x,y \in \RR^d$:
\begin{equation}
\| E(y)-E(x) \| \leq C \| y-x \|.
\end{equation}
\item[(E3)] There is a constant $a > 1/2$, s.t we obtain for the real part of the eigenvalues of $E(x)$:
\begin{equation}
\frac{1}{2}< a \leq \lambda(x)  \quad \text{ for all } x \in \RR^d.
\end{equation}
If the following property is satisfies in addition, we call $E(x)$ \emph{bounded admissible exponent}:
\item[(E4)] There is a constant $b \geq a$, s.t we obtain for the real part of the eigenvalues of $E(x)$:
\begin{equation}
 \Lambda(x) \leq b < \infty  \quad \text{ for all } x \in \RR^d.
\end{equation}
\end{enumerate}
\label{Def:Exponenten}
\end{definition}

If not mentioned otherwise, we will always consider admissible exponents in the remainder of the paper.

\begin{remark}
Property (E1) is equivalent to $E(x)$ being orthogonally diagonalizable. It is 
used e.g. in Proposition \ref{Prop:AbschMEExponent} to derive bounds for the matrix exponential.  Moreover, in this case one can always use the unit sphere $\SD$ in the desintegration formula \eqref{desint}. For the general case see Theorem 6.1.7 in \cite{Peter}.
The second property is needed in order to use the SDE techniques below. 
Without Property (E3) the process cannot exist. Even in trivial cases the integral under consideration would be infinite.
In the PhD-Thesis \cite{Daniel} it is shown that under assumption (E4) the resulting process is Feller. This is advantageous in considering some of the properties. Whenever possible, we will not make use of this assumption. 
\end{remark}

\begin{definition}
Let $E(x) \in GL(\RR^d), x \in \RR^d$, be an admissible exponent. We call  $\phi$ \emph{operator-stable-like L\'evy measure with exponent \bf{$E(x)$}}, if it admits the following representation with $A \in \calB(\Gamma)$ and $x \in \RR^d$: 
\begin{equation}\label{desint}
\phi_x(A):=\phi(x,A):= \int_{\SD} \int_0^{\infty} 1_A(r^{E(x)} \theta) \frac{dr}{r^2} \sigma(d\theta)
\end{equation}
where $ \sigma$ is a finite measure on $\SD$ ist.  We use  OSL L\'evy measure as a shorthand.
 \label{Defi:OperatorStableLikeLevyMass}
\end{definition}

\begin{remark}
\begin{itemize}
\item[(i)] The OSL L\'evy measure $\phi_x$ satisfies the following scaling property
\begin{equation}
  \left(t^{E(x)} \phi_x \right) (A) = t \phi_x(A)
\end{equation}  
for $ t >0$ and $ A \in \calB(\Gamma)$.
\label{Bem:SkalierungseigOSLLevyMass}
\item[(ii)] The measure $\phi_x$ is symmetric, that is, $\phi_x(A)= \phi_x(-A)$ for all $A \in \calB(\Gamma)$ and $x \in \RR^d$, if and only if $\sigma$ is symmetric.  
\end{itemize}
\end{remark}

In context of integration with respect to those measures, we obtain: 

\begin{lemma}
\label{Lemma:BererchnungIntegralOSLLevyMass}
For all measurable $f: \Gamma \rightarrow \RR_+$ it holds that:
\begin{equation}
\int_{\Gamma} f(y) \phi_x(dy) =\int_{\SD} \int_0^{\infty} f(r^{E(x)} \theta ) \frac{dr}{r^2} \sigma(d \theta).
\end{equation}
\end{lemma}

The lemma is proved by a standard monotone class argument. For the matrix exponential we obtain by the boundedness of the eigenvalues the following estimates which are essential for the remainder of the paper:

%
%

\begin{proposition}
\label{Prop:AbschMEExponent}
If $E(x)$ is an admissible exponent, there exists a constant $C>0$, such that
\begin{equation}
\| r^{E(x)} \| \leq C r^{a} \hspace{10mm} \text{ for } 0 <r<1.
\end{equation}
If $E(x)$ is even a bounded admissible exponent, there exists a constant $C>0$, such that additionally
\begin{equation}
\| r^{E(x)} \| \leq C r^{b}  \hspace{10mm}\text{ for } r \geq 1.
\end{equation} 
\end{proposition}

\begin{proof}
Since $E(x)$ is symmetric, this follows directly from Theorem \ref{Satz:AllgAbschDiagOperatorEW} which can be found in the appendix.
\end{proof}

By now, we have defined a state-space dependent family of L\'evy measures $(\phi_x(\cdot))_{x \in \RR^d}$. Subsequently, we will show that this family gives rise to a stochastic process. In search for SDEs having our class of processes as solutions, we use the SDE related to stable-like processes as a starting point (cf. Example  \ref{Bsp:StableLikeProzessLevyMass}).  Let us consider the $d$-dimensional stochastic differential equation
\begin{equation}  \label{SDE}
  X_t=X_0 + \int_0^t \int_{\SD} \int_0^1  r ^{E(X_{s-})}\theta \tilde{N}(ds,d \theta, dr) 
  +  \int_0^t \int_{\SD} \int_1^{\infty}  r ^{E(X_{s-})} \theta N(ds,d \theta, dr)
\end{equation}
where $E(x) \in GL(\RR^d), x  \in \RR^d$, is an admissible exponent. 
Let $N$ denote a Poisson random measure on the product space $\RR_+ \times \SD \times (0,\infty)$ which is adapted to the filtration $(\calF_t)_{t\geq 0}$ and having intensity measure $\lambda^1 \otimes \sigma \otimes \pi$. Here $\sigma$ is the finite measure on $\SD$ as in Definition \ref{Defi:OperatorStableLikeLevyMass} of the OSL L\'evy measure and  $\pi(dr)=r^{-2} dr$ on $(0, \infty)$. For the compensated Poisson random measure we write as usual $\tilde{N}$.  
 
 \begin{example}   \label{Bsp:StableLikeProzessLevyMass}
Choose in \eqref{SDE} for $ \sigma$ the uniform distribution $\SD$ and $E(x)= \frac{1}{\alpha(x)}id $ where $id$ is the identity matrix and $\alpha(x) \in C_b^1( \RR^d)$ Lipschitz continuous with 
\[
0 < \inf_{x \in \RR^d} \alpha(x)\leq \alpha(x) \leq \sup_{x \in \RR^d} \alpha(x)<2.
\]
We then obtain the SDE representation of $\alpha$-stable-like processes (cf. Proposition 2.1 in \cite{ChineseSDE}). Stable-like processes were first introduced by R. Bass in \cite{BassStableLike}.
\end{example}

In order to prove the following theorem we need an upper bound for the difference of two matrix exponentials (cf. Lemma \ref{Lem:AbschaetzungDiffExponenten} in the appendix).

\begin{theorem} \label{Satz:SDEeindeutigloesbar}
For each $\calF_0$ measurable $X_0$ there exists a unique strong solution of the stochastic differential equation \eqref{SDE}.
This solution is c\`adl\`ag and adapted. 
\end{theorem}

\begin{proof}
Due to the interlacing arguments in Theorem 6.2.9 in \cite{Apple} it is enough to consider the `reduced' SDE without big jumps
\begin{equation} \label{reducedSDE}
M_t= M_0 + \int_0^t \int_{\SD} \int_0^1 \ r^{E(M_{s-})} \theta \tilde{N}(ds, d \theta, dr).
\end{equation}
In order to show existence and uniqueness of this SDE we have to verify the Lipschitz condition and linear growth. 
Hence, it is enough to show that there exists constants $C_1,C_2>0$ such that for all $x,y \in \RR^d$
\begin{align}
\int_{\SD} \int_0^1 \|  r^{E(x)}  \theta \| ^2 \frac{dr}{r^2}  \sigma(d\theta) &\leq C_1 (1+ \|x\|^2), \\
\int_{\SD} \int_0^1 \|  r^{E(x)} \theta -  r^{E(y)}\theta \| ^2\frac{dr}{r^2}  \sigma(d\theta) &\leq 
C_2 \|x-y\|^2.
\end{align}
Linear growth follows from the estimate in Proposition \ref{Prop:AbschMEExponent} for $r \in (0,1)$  which yields the existence of a constant $C>0$ such that
\[
 \int_{\SD} \int_0^1 \|  r^{E(x)} \theta \| ^2 \frac{dr}{r^2}  \sigma(d\theta) \leq C \sigma(\SD) \int_0^1 r^{2a -2 }dr = C_1 \leq C_1(1+\|x\|^2).
 \]
Since the real parts of the eigenvalues of $E(x)$ are bounded from below by $a >\frac{1}{2}$, the integral is finite. Recall that we tacitly assume that the constant $\delta>0$ in Lemma \ref{Lem:AbschaetzungDiffExponenten} is chosen in a way that $ a- \delta > \frac{1}{2}$ is satisfied. The Lipschitz condition is obtained as follows: 
For $\delta >0$ there exists by Lemma \ref{Lem:AbschaetzungDiffExponenten} a constant $C>0$, such that
\begin{align*}
\int_{\SD} \int_0^1 \| r^{E(x)}\theta -  r^{E(y)}\theta \| ^{2}\frac{dr}{r^2}  \sigma(d\theta) &\leq 
\int_{\SD} \int_0^1 \| r^{E(x)} -  r^{E(y)} \|^2 \cdot \| \theta \|^2  \frac{dr}{r^2} \sigma(d\theta) \\
& \leq C \sigma(\SD) \|x-y\|^2  \int_0^1 r^{2a-2\delta-2} dr  \\
& \leq C_2 \|x-y\|^2,
\end{align*}
since $a- \delta > \frac{1}{2}$. \change{In particular, we have used \eqref{AbschaetzDiffExp} in order to obtain the second inequality}. 
\end{proof}



\begin{remark}
(a) It can be easily shown that the solution of the SDE \eqref{SDE} is a semimartingale. By the boundedness of the real parts of the eigenvectors, 
\begin{equation}
t \mapsto \int_0^t \int_{\SD} \int_0^1 r^{E(M_{s-})} \theta \tilde{N}(ds, d \theta, dr)
\end{equation}
is even an $(\calF_t)$-martingale in $L^2$. \\
(b) In the PhD-thesis \cite{Daniel} it is shown that under the additional hypothesis (E4) the solution is a Feller process and that the test functions $C_c^\infty(\bbr^d)$ are contained in the domain of the generator of the process (cf. Satz 5.18 and Satz 5.24 in the thesis).
Here, in order to be able to deal with a larger class of processes we use semimartingale techniques whenever possible.  
\end{remark}

Let us first show that the process is universal Markov, cf. Theorem 2.47 in \cite{Alex}. 

\begin{theorem}
The solution $X^x$ of the SDE \eqref{SDE} with starting point $x$  is a universal Markov process.
\end{theorem}

\begin{proof}
We consider the solution $M^x$ of the modified SDE \eqref{reducedSDE}. The extension to the general SDE is straight forward using interlacing. 
We have to show (MP3), that is, for all $u,t \geq 0, x \in \RR^d$ and $B \in \calB^d$ we have: 
\begin{equation}
  \PP^x \left(M_{u+t} \in B | \calF_{\change{u}}^M \right)= \PP^{M_u} \left(M_t \in B \right) \quad \PP^x \text{ -f.s. }
\label{Satz:SDEistUniMPEigenschaftUniMP}
\end{equation} 
Let $x,y \in \RR^d$ be fixed and $u \geq 0$ such that $M_u^y=x$.
Then under $ \PP^y \left( \cdot | M_u=x \right) $ we have
\begin{align*}
M_{u+t}&= y + \int_0^{u+t} \int_{\SD} \int_0^1 r^{E(M_{s-})} \theta \tilde{N}(ds, d \theta, dr) \\
&= y +  \int_0^u \int_{\SD} \int_0^1 r^{E(M_{s-})} \theta \tilde{N}(ds, d \theta, dr) +  \int_u^{u+t} \int_{\SD} \int_0^1 r^{E(M_{s-})} \theta \tilde{N}(ds, d \theta, dr) \\
&=x + \int_u^{u+t} \int_{\SD} \int_0^1 r^{E(M_{s-})} \theta \tilde{N}(ds, d \theta, dr) .
\end{align*}
By stationary increments of L\'evy processes we obtain on the other hand under $\PP^x$
\begin{align*}
M_t &=   x + \int_0^t \int_{\SD} \int_0^1 r^{E(M_{s-})} \theta \tilde{N}(ds, d \theta, dr) \\
&=  x + \int_u^{u+t} \int_{\SD} \int_0^1 r^{E(M_{s-})} \theta \tilde{N}(ds, d \theta, dr) 
\end{align*}
with solution $M_t^x$. Hence the result
\[
  \PP^y \left( M_{u+t} \in B | M_u=x \right) = \PP^x \left( M_t \in B \right).
\]
The measurability (MP1) follows directly from the adaptedness of the solution (cf.  Theorem \ref{Satz:SDEeindeutigloesbar}) and (MP2) is obviously satisfied since the process starts almost surely in $x$.
\end{proof}

By Theorem 3.33 in \cite{cinlarjacod}, we obtain that the solution is an It\^o process. There is a 1:1-correspondence between L\'evy processes and their characteristic exponent. In \cite{Schnurr} it was shown that the probabilistic symbol can be used as a generalization of the characteristic exponent in the more general framework of It\^o processes. 

\begin{definition} \label{def:symbol}
Let $X$ be an It\^o process, which is conservative and normal, that is,  $\bbp^x(X_0=x)=1$. Fix a starting point $x$ and define $\tau=\tau^x_k$ to be the first exit time from a compact neighborhood $K:=K_x$ of $x$: 
  \[
    \tau:=\inf\{t\geq 0 : X_t^x \notin K \}.
  \]
We call $p:\bbr^d\times \bbr^d\to \bbc$ given by
\begin{align} \label{stoppedsymbol} 
     p(x,\xi):=- \lim_{t\downarrow 0}\bbe^x \frac{e^{i(X^\tau_t-x)'\xi}-1}{t}
\end{align}
the \emph{symbol} of the process, if the limit exists and coincides for every choice of $K$.
\end{definition}

\begin{theorem}
The solution process of \eqref{SDE} is an It\^o process. Its symbol is:
\begin{equation}
\label{Formel:SymbolAnfangKap6}
q(x,\xi)= \int_{\SD} \int_0^{\infty} \left( 1- e^{i \langle  r^{E(x)}\theta, \xi \rangle} + 1_{\lbrace 0 < r <1 \rbrace} i \langle  r^{E(x)}\theta, \xi \rangle \right) \frac{dr}{r^2} \sigma(d \theta).
\end{equation}
\end{theorem}

\begin{proof}
This follows directly by \cite{Alex} Theorem 5.7.
\end{proof}

In the PhD-Thesis of D. Schulte, the following is shown in addition (c.f. \cite{Daniel} Satz 5.24 and Satz 5.18). We cite this here, since it is of some interest in its own right. 

\begin{theorem}
If $E(x)$ is a bounded admissible exponent, the solution of \eqref{SDE} is a rich Feller process, that is, a Feller process with the additional property that 
$C_c^\infty(\bbr^d) \subseteq D(A)$, where $D(A)$ denotes the (strong) domain of the generator $A$. 
\end{theorem}

By a direct calculation or by using well known facts on the symbol, we obtain the generator of the process. Under the additional assumption (E4), that is, in the Feller framework, it is the classical generator on $C^2_c(\bbr^d)$; in the more general case it is the extended generator (cf. \cite{Alex} Def. 1.17, \cite{vierleute} Def. 7.1). 
We recall the definition for the reader's convenience: 
\begin{definition} \label{def:extgen}
Let $X$ be a Markov process. A function $u\in B_b(\bbr^d)$ is said to belong to the \emph{domain of the extended generator} of $X$, written as $u\in D(A_{ext})$, if there exists a $w\in (\cb^d)^*$, the universally measurable functions, such that the process
\[
M_t^{[u]}=u(X_t)-u(X_0)-\int_0^t w(X_s) \ ds
\]
is well defined and a local martingale (see Section 2.1) for every $\bbp^x$. If we choose for every $u\in D(A_{ext})$ one $w$ with this property, we write $A_{ext} u:=w$ and call $A_{ext}$ (a version of) the \emph{extended generator} of $X$. 
\end{definition}

\begin{theorem}\label{Satz:BerechnungErzeugerGesamt}
The extended generator of the solution $X^{x}$ of the SDE \eqref{SDE} can be written (for $u \in C_b^2(\RR^d)$):
\begin{equation}
Au(x)=\int_{\SD} \int_0^{\infty} \left( u(x+ r ^{E(x)}\theta)- u(x) - 1_{ \lbrace 0 < r < 1\rbrace} \langle  r ^{E(x)}\theta , \nabla u(x) \rangle \right) \frac{dr}{r^2} \sigma(d \theta).
\end{equation}
If $(E4)$ is satisfied, the generator of the Feller process $X^{x}$ has the same representation (for $u \in C_b^2(\RR^d)$).
\end{theorem}

By now, we have only put restrictions on the exponent $E(x)$ and have allowed for arbitrary finite measures on the unit sphere  $\SD$
However, the aim of the present article is to construct processes whose $x$-dependent L\'evy triplet, that is, the differential characteristics,  are of the form $(0,0, \phi_x(\cdot))$. In other words we seek processes, whose symbol can be written by using the operator-stable-like L\'evy measure $\phi_x$, that is, 
\[
q(x,\xi)= \int_{\Gamma} \left( 1- e^{i \langle y, \xi \rangle} +1_{\lbrace \| y \| <1 \rbrace} i \langle y, \xi \rangle \right) \phi_x(dy)
\]
\change{where $\Gamma$ denotes $\RR^d \setminus \lbrace 0 \rbrace$}. Recall Lemma \ref{Lemma:BererchnungIntegralOSLLevyMass}, that is, for $f : \Gamma \rightarrow \RR_+$ we obtain
\begin{align} \label{intequation}
  \int_{\Gamma} f(y) \phi_x(dy) =\int_{\SD} \int_0^{\infty} f(r^{E(x)} \theta ) \frac{dr}{r^2} \sigma(d \theta).
\end{align}
Because of the indicator function in the integrand, generator and symbol can not be represented via the OSL L\'evy measure. In order to obtain this we have to pose an additional condition on $\sigma$, namely that  $\sigma$ is symmetric, that is, $\sigma(A)= \sigma(-A)$ for all $A \in \calB( \SD)$. 


By this symmetry assumption we obtain the following result. The proof which uses \eqref{intequation} is standard and hence omitted. 
\begin{theorem}
If $ \sigma$ is symmetric on $\SD$, the symbol \eqref{Formel:SymbolAnfangKap6} admits for $x,\xi \in \RR^d$ the representation:
\begin{equation}
q(x, \xi)=\int_{\SD} \int_0^{\infty} \left( 1- \cos( \langle r^{E(x)}\theta, \xi \rangle ) \right) \frac{dr}{r^2} \sigma(d \theta) =\int_{\Gamma} \left( 1- \cos( \langle y, \xi \rangle ) \right) \phi_x(dy) .
\end{equation}
\end{theorem}

The symmetry of  $\sigma$ is by Remark \ref{Bem:SkalierungseigOSLLevyMass} equivalent to the symmetry of the OSL L\'evy measure $\phi_x$.

\begin{definition}
Let $X=\left( X_t \right)_{t \geq 0}$ be a Markov semimartingale with symbol
\begin{equation}
q(x, \xi)=\int_{\Gamma} \left( 1- \cos( \langle y, \xi \rangle ) \right) \phi_x(dy) , \quad x \in \RR^d,
\end{equation}
where $\phi_x( \cdot)$ is a symmetric operator-stable-like L\'evy measure. We call $X$ \emph{operator-stable-like process}, OSL process, for short. 
\end{definition}

We have thus constructed an It\^o process
having the $x$-dependent L\'evy-triplet  $(0,0,\phi_x( \cdot))$ .
Recall that it is even a rich Feller process if the additional boundedness assumption (E4) is satisfied.

\begin{proposition}
The (extended) generator $A$ of an OSL process $X$ admits the representation
\begin{equation}
Au(x)= \int_{\Gamma} \big( u(x+y)-u(x) \big) \phi_x(dy) 
\end{equation}
for $u \in C_c^{\infty}(\RR^d)$.
\end{proposition}

\begin{example}\label{Bsp:BerechnungSymbolStableLikeProzess}
Let us shortly come back to the $\alpha$-stable-like process of Example \ref{Bsp:StableLikeProzessLevyMass}. 
Using $u=r^{1/ \alpha(x)} $ the L\'evy measure can be written as follows: 
\[
  \phi_x(dy)= \int_{\SD} \int_0^{\infty} 1_{dy} ( u \theta) \frac{\alpha(x)}{u^{1+\alpha(x)}} \, du \sigma(d \theta).
 \]
Since $ \sigma$ is the uniform distribution on $\SD$, we obtain
\[
  \phi_x(dy)= \frac{ C_{\alpha(x)}}{\| y \|^{d+\alpha(x)}} dy.
\]
The constant $C_{\alpha(x)}>0$ is chosen such that
\[
  \| \xi \|^{\alpha(x)} = C_{\alpha(x)} \int_{\Gamma} \left( 1- \cos( \langle y, \xi \rangle ) \right)\frac{dy}{\| y \|^{d+ \alpha(x)}}
\]
holds, that is, 
\[
C_{\alpha(x)} = \frac{\alpha(x) 2 ^{\alpha(x)-1} \Gamma \left( \frac{\alpha(x)+d}{2} \right)}{\pi^{d/2} \Gamma  \left( 1 - \frac{\alpha(x)}{2} \right)},
\]
cf. Exercise 18.23 in  \cite{KonstLevyMass}. The symbol of the stable-like process is hence
\[
q(x, \xi)= \int_{\Gamma} \left( 1- \cos \left( \langle y , \xi \rangle \right) \right) \phi_x(dy) = \| \xi \|^{\alpha(x)}.
\]
\end{example}

In order to derive properties of the symbol, we need the following estimate where $a$ is the lower bound for the real parts of the eigenvectors of $E(x)$ (cf. (E3)): 

\begin{lemma} \label{Lemma:AbschaetzungTaylorfuerSymbol}
There is a constant $C>0$, such that for all $x,\xi \in \RR^d$ and $ \theta \in \SD$:
\begin{equation}
\left| 1-e^{i \langle  r^{E(x)} \theta ,\xi \rangle}+i \langle r^{E(x)} \theta , \xi \rangle 1_{\lbrace 0 < r <1 \rbrace} \right| \leq C \left( 1+ \|\xi \|^2\right)  \left( 1 \wedge r^{2a} \right) .
\end{equation}
Additionally we have
\begin{equation}\label{intesti}
\int_0^{\infty} \left( 1 \wedge r^{2a}  \right) \frac{dr}{r^2} < \infty.
\end{equation}
\end{lemma}

\begin{proof}
Observe that
\begin{align*}
\left|1- e^{i \langle r^{E(x)} \theta , \xi \rangle } \right. &+ \left. i \langle  r^{E(x)} \theta , \xi\rangle 1_{\lbrace 0 < r <1 \rbrace} \right| \\
& \leq \left| \left( 1-e^{i \langle  r^{E(x)} \theta ,\xi \rangle}\right) 1_{\lbrace r \geq 1\rbrace } \right| + \left|  \left( 1-e^{i \langle  r^{E(x)} \theta , \xi \rangle}+i \langle r^{E(x)} \theta ,\xi \rangle \right) 1_{\lbrace 0 < r <1 \rbrace} \right| \\
& \leq  2 \cdot 1_{\lbrace r \geq 1 \rbrace}  +  \left| \langle r^{E(x)} \theta , \xi \rangle \right|^2 1_{\lbrace 0<r<1 \rbrace}  \\
& \leq   2 \cdot  1_{\lbrace r \geq 1\rbrace} +  \| r^{E(x)} \theta \|^2 \| \xi \|^2  1_{\lbrace 0<r<1 \rbrace}  \\
& \leq  2 \cdot  1_{\lbrace r \geq 1\rbrace}  +   \| r^{E(x)} \|^2 \| \theta \|^2 \| \xi \|^2  1_{\lbrace 0<r<1 \rbrace}  \\
& \leq  2 \cdot  1_{\lbrace r \geq 1\rbrace}  + \left( Cr^{2a} \| \xi \|^2  1_{\lbrace 0<r<1 \rbrace} \right) \\
& \leq  C \left( 1+ \|\xi \|^2 \right) \left( 1 \wedge r^{2a} \right),
\end{align*}
where we have used a Taylor expansion in the second summand, Cauchy-Schwarz inequality and the first inequality in Proposition \ref{Prop:AbschMEExponent}.
Since $a >1/2$, \eqref{intesti} follows.
\end{proof}

Subsequently we analyze properties of the symbol. These will help us to derive properties of the process. 

\begin{theorem} The symbol $q(x,\xi)$  of an OSL process has the following properties:
\begin{enumerate}
\item[(i)] For all  $x \in \RR^d$ we have $q(x,0)=0$;
\item[(ii)] The symbol is bounded, that is, there exists a constant $C>0$ such that 
\[
  \sup_{x \in \RR^d} |q(x, \xi) | \leq C \left( 1+ \| \xi \|^2 \right) \quad \text{ for all } \xi \in \RR^d ;
\]
\item[(iii)] $x \mapsto q(x, \xi)  \text{ is continuous for all } \xi \in \RR^d$;
\item[(iv)] The symbol $q(x, \cdot)$ is symmetric in $\xi$ for all $x \in \RR^d$.
\item[(v)] The symbol satisfies the so called sector condition, that is, there exists a constant $C>0$ with
\begin{align} \label{sector}
| \mathrm{Im}\, q(x,\xi) | \leq C \, \mathrm{Re} \, q(x, \xi) \quad \text{ for all } x,\xi \in \RR^d.
\end{align}
\end{enumerate}
\label{Satz:reelleSymbolEigenschaftenElementar}
\end{theorem}
\begin{proof}
Property (i) is obvious, since we only consider conservative processes. \\
Ad property (ii):
By Lemma \ref{Lemma:AbschaetzungTaylorfuerSymbol} we obtain
\begin{align*}
| q(x, \xi) | & \leq \int_{\SD} \int_0^{\infty} \left| e^{i \langle r^{E(x)} \theta , \xi \rangle}-1-i \langle r^{E(x)} \theta , \xi \rangle 1_{\lbrace 0 < r <1 \rbrace} \right|  \frac{dr}{r^2} \sigma(d \theta) \\
& \leq \int_{\SD} \int_0^{\infty} C \left( 1+ \|\xi \|^2\right)  \left( 1 \wedge r^{2a}  \right) \frac{dr}{r^2} \sigma(d \theta) \\
& \leq C \left( 1+ \|\xi \|^2\right).
\end{align*}
Ad property (iii):
This follows from Theorem 2.30 in \cite{LevyLike}, since continuity of  $x \mapsto q(x,0)$ it equivalent to continuity of $ x \mapsto q(x,\xi)$ for all  $\xi \in \RR^d$.
Properties (iv) and (v) hold by the fact that the symbol is real valued.  
\end{proof}

The following scaling property of the symbol of an OSL process is essential for analyzing many of its properties.

\begin{theorem}
The symbol $q(x,\xi)$  admits the following scaling property in the second argument for all $t >0$ 
\begin{equation}
q(x,t^{E(x)} \xi)=t q(x, \xi).
\end{equation}
\end{theorem}

\begin{proof}
By the symmetry of the exponent and the scaling property of the L\'evy measure $\phi_x( t^{-E(x)}A) =t \phi_x(A) $ for $ t>0$ for $ A \in \calB(\Gamma)$, cf. Remark \ref{Bem:SkalierungseigOSLLevyMass}:
\begin{align*}
q(x,t^{E(x)} \xi)&= \int_{\Gamma} \left( 1- \cos( \langle y, t^{E(x)} \xi \rangle) \right) \phi_x(dy) \\
&= \int_{\Gamma} \left( 1- \cos( \langle t^{E(x)} y,  \xi \rangle) \right) \phi_x(dy) \\
&= \int_{\Gamma} \left( 1- \cos( \langle y,  \xi \rangle) \right) \left(t^{E(x)}  \phi_x \right)(dy) \\
&= \int_{\Gamma} \left( 1- \cos( \langle y,  \xi \rangle) \right) t \phi_x( dy) \\
&= t q(x, \xi).
\end{align*}

\end{proof}

By the scaling property we get sharper bounds for the symbol.

\begin{theorem}\label{Satz:obereAbschSymbolReell}
Let $K$ be a compact set $ \RR^d$ and $q(x,\xi)$ be the symbol of an OSL process. In this case, there exist constants
 $C_1, \allowbreak C_2, \allowbreak C_3, \allowbreak C_4 >0$, such that
\begin{itemize}
\item[(i)] for all $x, \xi \in \RR^d$:
\begin{equation}
|q(x, \xi) | \leq 
\begin{cases}
C_1 \| \xi \|^{1/\Lambda(x)} & \text{ for } \| \xi \| \leq 1, \\
C_2 \| \xi \|^{1/\lambda(x)} & \text{ for } \| \xi \| \geq 1;
\end{cases}
\end{equation}
\item[(ii)] for all $x \in K$ and every  $ \xi \in \RR^d$: 
\begin{equation}
|q(x, \xi) | \geq 
\begin{cases}
C_3 \| \xi \|^{1/\lambda(x)} & \text{ for } \| \xi \| \leq 1, \\
C_4 \| \xi \|^{1/\Lambda(x)} & \text{ for } \| \xi \| \geq 1. 
\end{cases}
\end{equation}
\end{itemize}
\end{theorem}

\begin{proof}
The case  $\xi =0$ is because of $q(x,0)=0$ trivially satisfied. For $ \xi \in \Gamma$ we use the representation via generalized polar coordinates:
$$ \xi=\tau_x(\xi)^{E(x)} l_x(\xi),$$
where $\tau_x(\xi) >0$ and $l_x(\xi) \in \SD$.
By the scaling property of the symbol
\begin{equation}
q(x, \xi)= q(x, \tau_x(\xi)^{E(x)} l_x(\xi))= \tau_x(\xi) q(x,l_x(\xi)).
\label{Formel:DarstSymbolVerallPolarkoord}
\end{equation}

We first prove the upper bound.
By the boundedness of the symbol (cf. Theorem \ref{Satz:reelleSymbolEigenschaftenElementar}) there exists a constant  $C>0$ such that
$$ |q(x, \xi) | \leq C \left( 1+ \| \xi \|^2 \right) \quad \text{ for all } x, \xi \in \RR^d. $$
Hence, we derive for $\| \xi \| \leq 1$ by \eqref{Formel:DarstSymbolVerallPolarkoord} and the estimate for $\tau_x(\xi)$ in Lemma \ref{Lemma:AbschVerallgPolarkoordianten} (i):
$$ |q(x,\xi)| =\tau_x(\xi) | q(x, l_x(\xi))|  \leq \tau_x(\xi) C \left( 1 + \| l_x(\xi) \|^2 \right) = 2 C \tau_x(\xi)  \leq C_2 \| \xi \|^{1/\Lambda(x)}.$$
The case $\| \xi \| \geq 1$ follows analogously by Lemma \ref{Lemma:AbschVerallgPolarkoordianten} (ii).
For the lower bound we make use of the fact that the symbol $q(\cdot, \cdot)$ is continuous and positive on the compact set $K \times \SD$
Hence, there exists a constant $C>0$ such that for all $x \in K$ and $l_x(\xi) \in \SD$:
$$|q(x, l_x(\xi)| \geq C >0.$$
By the representation \eqref{Formel:DarstSymbolVerallPolarkoord} for the symbol and the lower bound for the growth of the radial component $\tau_x(\xi)$ (cf. Lemma \ref{Lemma:AbschVerallgPolarkoordianten}) the claim follows. 
\end{proof}

Using the boundedness assumptions (E3)  and (E4) of Definition \ref{Def:Exponenten}  we get weaker bounds having the advantage that they do not depend on $x$.

\begin{corollary} \label{Koro:AbschSymbolReell}
Let $q(x,\xi)$ be the symbol of an OSL process. 
Let $K$ be a compact set in $ \RR^d$. 
\begin{itemize}
\item[(i)] There exists a constant $C_1$ such that for all $x, \xi \in \RR^d$:
\begin{equation}
|q(x, \xi) | \leq   C_1 \| \xi \|^{1/a}  \text{ for } \| \xi \| \geq 1,
\end{equation}
and if in addition (E4) holds, there exists a constant $C_2$ s.t. for all $x, \xi \in \RR^d$:
\begin{equation}
|q(x, \xi) | \leq   C_2 \| \xi \|^{1/b}  \text{ for } \| \xi \| \leq 1.
\end{equation}
\item[(ii)]  There exists a constant $C_3$ s.t. for all $x \in K$ and all  $ \xi \in \RR^d$: 
\begin{equation}
|q(x, \xi) | \geq  C_3 \| \xi \|^{1/a}  \text{ for } \| \xi \| \leq 1,
\end{equation}
and if in addition (E4) holds, there exists a constant $C_4$ s.t. for all $x \in K$ and all  $ \xi \in \RR^d$: 
\begin{equation}
|q(x, \xi) | \geq  C_4 \| \xi \|^{1/b}  \text{ for } \| \xi \| \geq 1.
\end{equation}
\end{itemize}
\end{corollary}

\begin{remark}
Let us mention that, \change{by \eqref{Formel:DarstSymbolVerallPolarkoord}}, the symbol is because of
\[
  q(x,\xi) = \tau_x( \xi) q(x, l_x(\xi)), \quad x \in \RR^d \text{ and } \xi \in \Gamma,
\]  
entirely determined by its values on $\RR^d \times \left( \SD \cup\{ 0\} \right)$, where $0$ denotes the zero vector on $ \RR^d$.
\end{remark}

\section{Properties of Operator-stable-like Processes}

In the present section we analyze path properties and moment estimates of the new class of processes introduced in Section 2, that is It\^o processes having symbol
\[
  q(x, \xi)=\int_{\Gamma} \left( 1- \cos( \langle y, \xi \rangle ) \right) \phi_x(dy) , \quad x \in \RR^d,
\]
where $ \phi_x( \cdot)$ is a symmetric operator-stable-like L\'evy-meausre. 
As before, $E(x)$ is symmetric, Lipschitz continuous and for the real parts of the  eigenvalues we have: 

\begin{align} \tag{E3}
  \frac{1}{2}< a \leq \lambda(x) \quad \text{ for all } x \in \RR^d
\end{align}
with constant $a > 1/2$. Sometimes we will have to demand in addition
\begin{align} \tag{E4}
   \Lambda(x) \leq b < \infty  \quad \text{ for all } x \in \RR^d
\end{align}
with constant $b \in \bbr_+$. 

Whenever possible, we will work in the more general and hence more flexible framework only demanding (E1) - (E3).  
Subsequently, we consider maximal inequalities, existence of moments, limits at zero and infinity as well as $p$-variation. 
Compare in this context Chapter 5 of \cite{LevyLike}.


\subsection{Maximal inequalities}

For every $x \in \RR^d$ and $R>0$ we denote the first exit time $T$ of the process $X=(X_t)_{t \geq 0}$ from the ball $ \overline{B}_R(x)$ by
$T :=T_R^x  := \inf \lbrace  t>0 : \| X_t-x \| >R \rbrace$. This exit time is closely related to the the supremum of the norm of the process via
\begin{equation}
\lbrace T < t \rbrace \subset \left\lbrace \sup\limits_{s \leq t} \| X_s-x \|>R \right\rbrace \subset \lbrace T \leq t \rbrace \subset \left\lbrace \sup\limits_{s \leq t} \| X_s-x \| \geq R \right\rbrace
\label{Formel:BezErstaustrittszeitMaxProz}
\end{equation}

We will make use of the following general result (c.f. \cite{Schnurr}, Proposition 3.10)
, which we recall for the readers' convenience. 

\begin{theorem}
Let $X$ be an It\^o process with continuous differential characteristics; in case of the measure $N$ this is meant in the sense that the function  $n:=\int_{y\neq 0} (1\wedge \norm{y}^2) \ N(\cdot,dy)$ is continuous. 

Then, there exists a constant  $C>0$ (only depending on the dimension) such that
\begin{equation}
\mathbb{P}^x \left( \sup\limits_{s \leq t}\|X_s-x\| > R \right)  \leq \mathbb{P}^x(T \leq t) \leq Ct \sup\limits_{\|y-x\| \leq R} \sup\limits_{\|\xi\| \leq 1/R} |q(y,\xi)| 
\label{Formel:Satz:FellerProzessErstaustrittszeit}
\end{equation}
for all $x \in \RR^d$ and $R,t>0$.
\label{Satz:FellerProzessErstaustrittszeit}
\end{theorem}

Analogously to the proof of  Corollary 5.3 in \cite{LevyLike} we can derive the following:

\begin{corollary} \label{Koro:UntereGrenzeEWErsteintrittszeit}
Let $X$ be an It\^o process with continuous differential characteristics as above. Then:
\begin{equation}
  \EE^x (T) \geq \frac{C}{\sup\limits_{\|y-x\| \leq R} \sup\limits_{\|\xi\| \leq 1/R} |q(y,\xi)|} \quad \text{ for all } x \in \RR^d \text{ and } R>0
\end{equation}
with constant $C>0$ of \eqref{Formel:Satz:FellerProzessErstaustrittszeit}.
\end{corollary}

For the case of OSL processes we obtain: 

\begin{corollary}\label{Koro:MaximalabschReellesSymbol}
Let $X$ be an OSL process. Then, there exists $C>0$, such that
\begin{enumerate}
\item[(i)] for all $x \in \RR^d$, $t>0$ and $0 < R \leq 1$:
\begin{align*}
  \mathbb{P}^x \left( \sup\limits_{s \leq t}\|X_s-x\| > R \right) & \leq   C t R^{-1/a}.
\end{align*}
\item[(ii)] If (E4) is satisfied in addition, then there exists $C>0$, such that for all $x \in \RR^d$, $t>0$ and $R \geq 1$:
\begin{align*}
  \mathbb{P}^x \left( \sup\limits_{s \leq t}\|X_s-x\| > R \right) & \leq  C t R^{-1/b}.
\end{align*}
\end{enumerate}
The same upper bound holds for  $ \PP^x \left( T \leq  t \right) $.
\end{corollary}

\begin{proof}
We will make use of the upper bound in Corollary \ref{Koro:AbschSymbolReell}, that it, there exist $C_1,C_2>0$ such that
\[ |q(x, \xi) | \leq 
  \begin{cases}
    C_1 \| \xi \|^{1/a} & \text{ for } \| \xi \| \geq 1
 \\    C_2 \| \xi \|^{1/b} & \text{ for } \| \xi \| \leq 1 
   \end{cases}
\]
for all $x \in \RR^d$. Set $ \tilde{C} :=  \max \lbrace C_1,C_2 \rbrace$. Here, and in the remainder of the proof, we only consider $b$ in the case where (E4) is satisfied. We obtain for all $x \in \RR^d$ and $R>0$:
\begin{align*}
  \sup\limits_{\| x-y \| \leq R} \sup\limits_{\| \xi \| \leq 1/R } | q(y, \xi) | \leq \sup\limits_{\| \xi \| \leq 1/R }
  \begin{cases}
    \tilde{C} \| \xi \|^{1/a} & \text{ for } \| \xi \| \geq 1\\
    \tilde{C} \| \xi \|^{1/b} & \text{ for } \| \xi \| \leq 1.
  \end{cases}
  \tag{$\ast$}
  \label{Stern:FellerMaxiumsprozess}
\end{align*}
Let us now consider $R \geq 1$ and hence the supremum over $\| \xi \| \leq 1/R \leq 1$. Then
$$ \eqref{Stern:FellerMaxiumsprozess} =  \sup\limits_{\| \xi \| \leq 1/R } \tilde{C} \| \xi \|^{1/ b} = \tilde{C} R^{-1/b}.$$
\\
for $ 0 <R \leq 1$ we obtain
$$ \eqref{Stern:FellerMaxiumsprozess} =  \sup\limits_{\| \xi \| \leq 1/R } \tilde{C} \| \xi \|^{1/ a}
= \tilde{C} R ^{-1 / a}.$$
The claim follows by Theorem \ref{Satz:FellerProzessErstaustrittszeit}.
\end{proof}

Combining the previous three results, we obtain for the first exit time: 
\begin{corollary}
Let $X$ be an OSL process. Then, there is a constant $C>0$, such that for all
 $x \in \RR^d$:
\begin{equation}
  \EE^x (T)
  \geq 
  \begin{cases}   
    C R^{1/a} & \text{ for } R \leq 1\\ 
    C R^{1/b} & \text{ for } R  \geq 1
  \end{cases}
\end{equation}
where the second inequality only holds, if (E4) is satisfied. 
\end{corollary}

The following counterpart of Theorem \ref{Satz:FellerProzessErstaustrittszeit} yields an upper bound for the tail probabilities of the first exit time. 
In order to make use of the following finer result, we demand $(E4)$ and work in the Feller framework. 
Let us recall first  Theorem 5.5 of \cite{LevyLike}.

\begin{theorem} \label{Satz:WAbschaetzungNachUnten}
Let $X$ be a Feller process with generator $(A,D(A))$, symbol $q(x,\xi)$ and $C_c^{\infty}(\mathbb{R}^d) \subset D(A)$. 
Then for all $x \in \mathbb{R}^d$ and $ R,t >0$:
\begin{equation}
  \mathbb{P}^x \left( T  \geq t\right) \leq C \left( t \sup\limits_{\|\xi \| \leq 1/(Rk(x,R))} \ \inf\limits_{\| y-x\| \leq R} \mathrm{Re} \ q(y,\xi) \right)^{-1} 
\end{equation}
as well as
\begin{equation}
  \mathbb{P}^x \left( \sup\limits_{s \leq t} \|X_s-x\|   \leq R \right) \leq C \left( t \sup\limits_{\|\xi \| \leq 1/(Rk(x,R))} \ \inf\limits_{\| y-x\| \leq R} \mathrm{Re} \ q(y,\xi) \right)^{-1} 
\end{equation}
with constant $C>0$ and
\begin{equation} \label{SatzPfadEigenschaft2Definitionk(x,r)}
  k(x,R):= \inf \left\lbrace k \geq \left( \arccos \sqrt{2/3}\right)^{-1} : \inf\limits_{\|\xi\| \leq 1/(kR) }\inf\limits_{\|y-x\| \leq R} \frac{\mathrm{Re}\ q(y,\xi)}{\|\xi\| \ |  \mathrm{Im} \ q(x,\xi)|} \geq 2R \right\rbrace, 
\end{equation}
if $\mathrm{Im} \ q(x,\xi) \not\equiv 0$, resp. $k(x,R)= \left( \arccos \sqrt{2/3}\right)^{-1}$, if $\mathrm{Im} \ q(x,\xi) \equiv 0$.
\end{theorem}

\begin{remark}
It can be shown that, if the sector condition \eqref{sector} is satisfied, the theorem above can be simplified. In this case, the constant $k(x,R)$ is bounded as follows: 
\[
  k(x,R) \leq \max \left\lbrace 2 C, \left( \arccos \sqrt{2/3}\right)^{-1} \right\rbrace.
\]
\end{remark}

In addition one can derive an upper bound for the first exit time $T$ (c.f. Corollary 5.8 of \cite{LevyLike}).

\begin{corollary}
Under the assumptions of Theorem \ref{Satz:WAbschaetzungNachUnten} we have for $x \in \RR^d$ and $R>0$
$$ \EE^x(T) \leq C \left( \sup\limits_{\|\xi\| \leq 1/(Rk(x,R))} \inf\limits_{\|y-x\| \leq R} \mathrm{Re} \ q(y,\xi) \right)^{-1} $$
with constant $C>0$ and $k(x,R)$ of the previous theorem.
\end{corollary}

The previous Theorem can be used in particular on real valued symbols satisfying locally the estimate
$| q(x,\xi)| \geq C \| \xi \|^{\alpha}$ with $\alpha \in (0,2)$ and a constant $C>0$, hence for OSL  processes satisfying (E4): 

\begin{corollary} \label{Koro:MaximalabschReellesSymbolUntere}
Let $X$ be an OSL process satisfying (E4). Then, there exist $C_1,C_2>0$, such that
\begin{enumerate}
  \item[(i)] for all $x \in \RR^d$, $t>0$ and $R \geq \frac{1}{k_0}$:
\begin{align*}
  \mathbb{P}^x \left( \sup\limits_{s \leq t}\|X_s-x\| \leq  R \right) & \leq  \frac{C_1}{t} R^{1 / a };
\end{align*}
  \item[(ii)]  for all $x \in \RR^d$, $t>0$ and $0 < R \leq \frac{1}{k_0}$:
\begin{align*}
  \mathbb{P}^x \left( \sup\limits_{s \leq t}\|X_s-x\| \leq R \right) & \leq   \frac{C_2}{t}   R^{1 / b},
\end{align*}
\end{enumerate}
where $k_0= \left( \arccos \sqrt{2/3}\right)^{-1}$.
The same estimates hold for $ \PP^x \left( T \geq  t \right) $.
\end{corollary}

\begin{proof}
The proof is similar to Corollary \ref{Koro:MaximalabschReellesSymbol}. However, here we use \ref{Satz:WAbschaetzungNachUnten} and bound the symbol from below. 
\\
Since the symbol of the OSL process is real valued, the constant in Theorem \ref{Satz:WAbschaetzungNachUnten} is
 \[
   k_0:=k(x,R)= \left( \arccos \sqrt{2/3}\right)^{-1}.
 \]  
By Corollary \ref{Koro:AbschSymbolReell} there exist constants $C_3,C_4 >0$, such that for all $x \in K \subset \RR^d$ ($K$ compact set) and all $\xi \in \RR^d$:
\[
  |q(x, \xi) | \geq 
  \begin{cases}
    C_3 \| \xi \|^{1/a} & \text{ for } \| \xi \| \leq 1, \\
    C_4\| \xi \|^{1/b} & \text{ for } \| \xi \| \geq 1. 
  \end{cases}
\]
Now let $ \tilde{C} = C_3 \wedge C_4 $. Then for all $R >0$:
\[
  \sup\limits_{\|\xi \| \leq 1/(Rk_0)} \ \inf\limits_{\| y-x\| \leq R} | \ q(y,\xi)|   \geq
  \sup\limits_{\|\xi \| \leq 1/(Rk_0)}
  \begin{cases}
    \tilde{C} \| \xi \|^{1/a} & \text{ for } \| \xi \| \leq 1, \\
    \tilde{C} \| \xi \|^{1/b} & \text{ for } \| \xi \| \geq 1. 
  \end{cases}
\]
Let us first consider $R \geq \frac{1}{k_0}$, hence $ \frac{1}{Rk_0} \leq 1$. This lower bound yields with Theorem \ref{Satz:WAbschaetzungNachUnten}:
\begin{align*}
  \mathbb{P}^x \left( \sup\limits_{s \leq t}\|X_s-x\| \leq  R \right) &\leq C \left( t \sup\limits_{\|\xi \| \leq 1/(Rk_0)} \tilde{C} \| \xi \|^{1/a} \right)^{-1} \\
  &= C \left( t \tilde{C} ( R k_0)^{-1/a} \right)^{-1} \\
  &= \frac{C_1}{t} R^{1/ a}
\end{align*}
and hence the result.
\\
The case $ 0 < R \leq \frac{1}{k_0}$ works analogously, but we have to consider $\frac{1}{R k_0} \geq 1$:
\begin{align*}
  \mathbb{P}^x \left( \sup\limits_{s \leq t}\|X_s-x\| \leq  R \right) &\leq C \left( t \sup\limits_{\|\xi \| \leq 1/(Rk_0)} \tilde{C} \| \xi \|^{1/b} \right)^{-1} \\
  &= C \left( t \tilde{C} ( R k_0)^{-1/b} \right)^{-1} \\
  &= \frac{C_2}{t} R^{1/ b}.
\end{align*}

\end{proof}

By the same reasoning as \ref{Koro:UntereGrenzeEWErsteintrittszeit}
we obtain as well an upper bound for the expected value of the exit time. 

\begin{corollary}
Let $X$ be an OSL process satisfying (E4). Then, there are constants $C_1,C_2 >0$, such that for all $x \in \RR^d$:
\[ \EE^x(T) \leq 
  \begin{cases}
    C_ 1  \cdot R^{1/b} & \text{ for } 0 < R \leq 1/k_0 \\
    C_2 \cdot R^{1/a} & \text{ for } R \geq 1/k_0
  \end{cases}
\]
with $k_0= \left( \arccos \sqrt{2/3}\right)^{-1}$.
\end{corollary}


\subsection{Moments}

By the maximal inequality from above we derive the existence of moments of the maximal process. In the present subsection we always have to assume that (E4) holds. 

\begin{theorem} \label{Satz:MaximumsprozessMomente}
Let $X$ be an OSL process satisfying (E4). 
For all $0 < p < \frac{1}{b}$ we obtain:
\[
  \EE^x \left(  \left(  \sup\limits_{s \leq t} \| X_s -x \|  \right)^p \right) < \infty.
\]
\end{theorem}

\begin{proof}
The proof goes along the same lines as the proof of Theorem 6.8 in \cite{Alex}.
By a well known formula for the expected value of positive random variables, we obtain for  $0 < p < \frac{1}{b}$
\begin{align*}
 \EE^x &\left( \left( \sup\limits_{s \leq t} \| X_s -x \| \right)^{p} \right) = p \int_0^{\infty} y^{p-1} \cdot  \PP^x \left( \sup\limits_{s \leq t} \| X_s -x \| > y  \right) \, dy \\
 &= p \int_0^1 y^{p-1}  \cdot \PP^x \left( \sup\limits_{s \leq t} \| X_s -x \| > y  \right) \, dy + p \int_1^{\infty} y^{p-1} \cdot  \PP^x \left( \sup\limits_{s \leq t} \| X_s -x \| > y  \right) \, dy \\
 & \leq p \int_0^1 y^{p-1}   \, dy + p \int_1^{\infty} y^{p-1} \cdot  \PP^x \left( \sup\limits_{s \leq t} \| X_s -x \| > y  \right) \, dy \\
& \leq 1 + p \cdot C \cdot t \int_1^{\infty} y^{p-1-\frac{1}{b}} \, dy < \infty
\end{align*}
with a constant $C>0$.
The last estimate follows from Corollary \ref{Koro:MaximalabschReellesSymbol}.
\end{proof}

Now we consider asymptotic bounds of fractional moments of the OSL processes, that is, we consider the short- and long-time behaviour of 
\[
   \EE^x \left( \sup\limits_{s \leq t} \| X_s -x \|^{p} \right) \quad \text{ for } p>0.
\]
For $t \in (0,1)$ we cannot hope for a better result than
$$ \EE^x \left( \sup\limits_{s \leq t} \| X_s -x \|^{p} \right) \leq Ct$$
with a constant $C>0$. Otherwise by the Kolmogorov-Chentsov theorem the existence of a continuous modification would follow. 
The next proof follows an idea of \cite{MomenteFranzKuehn}. 

\begin{theorem}
Let $X$ be an OSL process satisfying (E4).
 Then there exists a constant $C>0$ such that
 for all $t \geq 1$ and $0 < p < \frac{1}{b}$:
\begin{equation}
  \EE^x \left( \sup\limits_{s \leq t} \| X_s-x \|^p \right) \leq C t^{b p}.
\end{equation}
\end{theorem}

\begin{proof}
By Corollary \ref{Koro:MaximalabschReellesSymbol} for all $t \geq 1$ and $0 < p < \frac{1}{b}$ :
\begin{align*}
\EE^x \left( \sup\limits_{s \leq t} \| X_s-x \|^p \right) &= \int_0^{\infty} \PP^x \left( \sup_{s \leq t} \| X_s-x \| > y^{1/p} \right)  \, dy \\
& \leq \int_0^{t^{b p}} 1 \, dy +  \int_{t^{b p}}^{\infty} \PP^x \left( \sup_{s \leq t} \| X_s-x \| > y^{1/p} \right) \, dy \\
&= \int_0^{t^{b p}} 1 \, dy +  C_1 \cdot t \cdot  \int_{t^{b p}}^{\infty} y ^{-\frac{1}{b p}} \, dy \\
&=  t^{b p}  + C_2 \cdot t \cdot t^{b p- 1} = C t^{b p}
\end{align*}
with constants $C,C_1,C_2>0$.
\end{proof}

The analogous result for the short time behaviour reads: 
\begin{theorem}
Let $X$ be an OSL process satisfying (E4). Then there exists a constant $C>0$, such that 
for all $t \in (0,1)$ and $0 < p < \frac{1}{b}$:
\begin{equation}
  \EE^x \left( \sup\limits_{s \leq t} \| X_s-x \|^p \right) \leq C \left(  t^{a p}+ t \right).
\end{equation}
\end{theorem}

\begin{proof}
By \ref{Koro:MaximalabschReellesSymbol} for all $t \in (0,1)$ and $0 < p < \frac{1}{b}$ :
\begin{align*}
\EE^x \left( \sup\limits_{s \leq t} \| X_s-x \|^p \right) &= \int_0^{\infty} \PP^x \left( \sup_{s \leq t} \| X_s-x \| > y^{1/p} \right) dy \\
& \leq \int_0^{t^{ap}} 1 dy + \int_{t^{a p}}^{1} \PP^x \left( \sup_{s \leq t} \| X_s-x \| > y^{1/p} \right) dy \\
& \qquad + \int_1^{\infty} \PP^x \left( \sup_{s \leq t} \| X_s-x \| > y^{1/p} \right) dy \\
& \leq t^{a p} +  C_1 \cdot t \cdot  \int_{t^{a p}}^{1} y^{- \frac{1}{a p}}  dy  + C_2 \cdot t \cdot \int_1^{\infty} y^{-\frac{1}{b p}} dy \\
&= t^{ap} + \tilde{C_1} t \left( 1 -  t^{ap-1} \right) - \tilde{C_2} t  \\
&= C \left( t^{ap} + t \right).
\end{align*}
Recall that  $0 < ap < 1$ and $b p<1$.
\end{proof}


\subsection{Asymptotic Properties}

In what follows we would like to derive local and global growth properties for the paths of OSL processes. We show for which $ \gamma> 0$ 
\[
\varliminf\limits_{t }  \frac{\sup_{s \leq t}\|X_s-x\|}{t^{1/ \gamma}} \quad \text{ or } \quad \varlimsup\limits_t \frac{\sup_{s \leq t}\|X_s-x\|}{t^{1/ \gamma}}
\]
is finite or infinite under $\PP^x$ for $t \rightarrow 0$ repectively $t \rightarrow \infty$. Let us first deal with local properties, that is, growth at zero. To this end we use the generalized Blumenthal-Getoor indices (cf. Definition 5.13 in \cite{LevyLike} and Definition 3.8 in \cite{Schnurr}).

\begin{definition}
Let $q(x,\xi)$ be a negative definite symbol. Then the \emph{generalized Blumenthal-Getoor indices at infinity} are
\begin{align*}
\beta^x_{\infty} &:= \inf \left\lbrace  \gamma >0 : \lim\limits_{\| \xi\| \rightarrow \infty} \frac{\sup_{\| \eta \| \leq \| \xi \|} \sup_{\|y-x \| \leq 1/ \| \xi \|} | q(y,\eta)|}{\| \xi \| ^{\gamma}}= 0 \right\rbrace, \\
\underline{\beta}^x_{\infty} &:= \inf \left\lbrace  \gamma >0 : \varliminf\limits_{\| \xi\| \rightarrow \infty} \frac{\sup_{\| \eta \| \leq \| \xi \|} \sup_{\|y-x \| \leq 1/ \| \xi \|} | q(y,\eta)|}{\| \xi \| ^{\gamma}}= 0 \right\rbrace, \\
\overline{\delta}^x_{\infty} &:= \sup \left\lbrace  \gamma >0 : \varlimsup\limits_{\| \xi\| \rightarrow \infty} \frac{\inf_{\| \eta \| \geq \| \xi \|} \inf_{\|y-x \| \leq 1/ \| \xi \|} \mathrm{Re} \ q(y,\eta)}{\| \xi \| ^{\gamma}}= \infty \right\rbrace, \\
\delta^x_{\infty} &:= \sup \left\lbrace  \gamma >0 : \lim\limits_{\| \xi\| \rightarrow \infty} \frac{\inf_{\| \eta \| \geq \| \xi \|} \inf_{\|y-x \| \leq 1/ \| \xi \|} \mathrm{Re} \ q(y,\eta)}{\| \xi \| ^{\gamma}}= \infty \right\rbrace. 
\end{align*}
\end{definition}


Between the indices the following inequalities hold
\[ 
  0 \leq \delta^x_{\infty} \leq \underline{\beta}^x_{\infty} \leq \beta^x_{\infty} \leq 2 \quad \text{ and } \quad 0 \leq \delta^x_{\infty} \leq \overline{\delta}^x_{\infty} \leq \beta^x_{\infty} \leq 2.
\]  
Indices of stable-like processes have been investigated in Example 5.5 of \cite{Pfad}.
\begin{example}
For the symbol  of a stable-like process $q(x,\xi)= \| \xi \|^{\alpha(x)}$ the indices are:
\[
  \delta^x_{\infty} = \overline{\delta}^x_{\infty}=\underline{\beta}^x_{\infty} = \beta^x_{\infty} = \alpha(x).
\]
\end{example}

In general the indices are not equal. Using the indices at infinity one can derive the short-time behavior of path of Feller processes. This usually depends on the starting point $x$. 

\begin{theorem}
Let $X$ be an It\^o process with symbol  $q(x,\xi)$. \\
Then it holds $\mathbb{P}^x$-a.s.
\begin{align}
\lim\limits_{t \rightarrow 0} \frac{\sup_{s \leq t}\|X_s-x\|}{t^{1/ \gamma}}&=0 \quad	\text{ for all } \gamma > \beta^x_{\infty} ,
\label{Formel:SatzPfadEigenschaft1.Absch}\\
\varliminf\limits_{t \rightarrow 0} \frac{\sup_{s \leq t}\|X_s-x\|}{t^{1/ \gamma}}&=0 \quad	\text{ for all } \gamma > \underline{\beta}^x_{\infty}
\label{Formel:SatzPfadEigenschaft2.Absch}
\intertext{ and if the sector condition \eqref{sector} is satisfied,}
\varlimsup\limits_{t \rightarrow 0} \frac{\sup_{s \leq t}\|X_s-x\|}{t^{1/ \gamma}}&=\infty \quad	\text{ for all } \gamma < \overline{\delta}^x_{\infty}, 
\label{Formel:SatzPfadEigenschaft3.Absch} \\
\lim\limits_{t \rightarrow 0} \frac{\sup_{s \leq t}\|X_s-x\|}{t^{1/ \gamma}}&=\infty \quad	\text{ for all } \gamma < \delta^x_{\infty}.
\label{Formel:SatzPfadEigenschaft4.Absch}
\end{align}
\label{Satz:PfadEigenschaft0} 
\end{theorem}

\begin{proposition}
The OSL process admits the indices
\[
 \beta^x_{\infty} = \underline{\beta}^x_{\infty}= \frac{1}{\lambda(x)} \quad \text{ and } \quad \overline{\delta}^x_{\infty}=\delta^x_{\infty} = \frac{1}{\Lambda(x)}.
\]
\end{proposition}

\begin{proof}
Let us first derive $ \beta^x_{\infty}$:

Because of Theorem \ref{Satz:obereAbschSymbolReell} there exists a constant $C>0$, such that for all $y,\eta \in \RR^d$ with $ \| \eta \| \geq 1$:
\[
   | q(y,\eta) | \leq C \| \eta \| ^{1/ \lambda(y)}.
\] 
Furthermore, 
\[
  C \sup_{\| \eta \| \leq \| \xi \|} \sup_{\|y-x \| \leq 1/ \| \xi \|} \| \eta \| ^{1/ \lambda(y)} = C \sup_{\|y-x \| \leq 1/ \| \xi \|} \| \xi \| ^{1/ \lambda(y)}.
\]  
By the Lipschitz continuity of $E(x)$ the eigenvalues are continuous, too. Hence, 
\[
 \lim_{n \rightarrow \infty} \sup_{\|y-x \| \leq 1/n } \| \xi \| ^{1/ \lambda(y)} = \lim_{\| y \| \rightarrow \| x \|} \| \xi \|^{1/ \lambda(y)} = \| \xi \| ^{1 / \lambda(x)}.
\] 
We arrive at
\[
\beta^x_{\infty}= \inf \left\lbrace  \gamma >0 : \lim\limits_{\| \xi\| \rightarrow \infty} \frac{\sup_{\| \eta \| \leq \| \xi \|} \sup_{\|y-x \| \leq 1/ \| \xi \|} | q(y,\eta)|}{\| \xi \| ^{\gamma}}= 0 \right\rbrace= \frac{1}{\lambda(x)}.
\]
Analogously for $\underline{\beta}^x_{\infty}$.
\\
Now we come to $ \delta^x_{\infty} $: Since the sector condition holds, we can replace  $\mathrm{Re} \ q(\cdot, \cdot)$ by $|q(\cdot, \cdot)|$ in the definition of the index under consideration. By Theorem \ref{Satz:obereAbschSymbolReell} there exists a constant $C>0$, such that for all $\eta \in \RR^d$ with $\| \eta \| \geq 1$ and all $y \in K$, where $K \subset \RR^d$ is compact,
\[
  | q(y,\eta) | \geq C \| \eta \| ^{1 / \Lambda(y)}.
\]
Furthermore,
\[
  C \inf_{\| \eta \| \geq \| \xi \|} \inf_{\|y-x \| \leq 1/ \| \xi \|} \| \eta \|^{1 / \Lambda(y)}= C \inf_{\|y-x \| \leq 1/ \| \xi \|} \| \eta \|^{1 / \Lambda(y)}
\]
and by continuity of eigenvalues we obtain
\[
   \lim_{n \rightarrow \infty} \inf_{\|y-x \| \leq 1/ n}  \| \eta \|^{1 / \Lambda(y)} = \| \xi \|^{1/ \Lambda(x)}.
 \]
Summing up
\[
\delta^x_{\infty} := \sup \left\lbrace  \gamma >0 : \lim\limits_{\| \xi\| \rightarrow \infty} \frac{\inf_{\| \eta \| \geq \| \xi \|} \inf_{\|y-x \| \leq 1/ \| \xi \|} \mathrm{Re} \ q(y,\eta)}{\| \xi \| ^{\gamma}}=  \infty \right\rbrace = \frac{1}{\Lambda(x)}.
\]
Analogoulsly for $\overline{\delta}^x_{\infty}$.
\end{proof}

The short-time behavior can be derived by Theorem \ref{Satz:PfadEigenschaft0} . Since we are dealing with a local property, the boundedness assumption (E4) is not needed. However, compare Corollary \ref{cor:longterm}. Even if we proved that this result remains true for the unbounded case, the infimum would be zero and the result would no longer be useful.

\begin{corollary}
Let $X$ be an OSL process.
It holds $\PP^x$-a.s.:
\begin{align*}
\lim\limits_{t \rightarrow 0} \frac{\sup_{s \leq t}\|X_s-x\|}{t^{1/ \gamma}}&=0 \quad	\text{ for all } \gamma > \frac{1}{\lambda(x)},\\
\lim\limits_{t \rightarrow 0} \frac{\sup_{s \leq t}\|X_s-x\|}{t^{1/ \gamma}}&=\infty \quad	\text{ for all } \gamma < \frac{1}{\Lambda(x)} .
\end{align*}
\end{corollary}

In order to analyze the long-time behavior of the paths we consider the indices at zero 
(cf. Definition 5.17 in \cite{LevyLike} and \cite{Schnurr} Definition 3.8). Here, we assume (E4), therefore, the processes we are dealing with are Feller.

\begin{definition}
Let $q(x,\xi)$ be a negative definite symbol with bounded coefficients, that is,  $\sup_{x \in \mathbb{R}^d} |q(x,\xi)| \leq C \left( 1+ \| \xi \|^2 \right)$  with constant $C>0$. The \emph{ generalized Blumenthal-Getoor indices (at zero)} are
\begin{align*}
\beta_0:= & \sup \left\lbrace \gamma \geq 0 : \lim\limits_{\| \xi \| \rightarrow 0} \frac{\sup_{x \in \mathbb{R}^d} \sup_{\| \eta \| \leq \| \xi \|} | q(x,\eta)|}{\| \xi \|^{\gamma}}=0 \right\rbrace, \\
\underline{\beta}_0:= & \sup \left\lbrace \gamma \geq 0 : \varliminf\limits_{\| \xi \| \rightarrow 0} \frac{\sup_{x \in \mathbb{R}^d} \sup_{\| \eta \| \leq \| \xi \|} | q(x,\eta)|}{\| \xi \|^{\gamma}}=0 \right\rbrace, \\
\overline{\delta}_0 := & \inf \left\lbrace \gamma \geq 0 : \varlimsup\limits_{\| \xi \| \rightarrow 0} \frac{\inf_{x \in \mathbb{R}^d} \inf_{\| \eta \| \geq \| \xi \|} \mathrm{Re} \ q(x,\eta)}{\| \xi \|^{\gamma}}=\infty \right\rbrace, \\
\delta_0 := & \inf \left\lbrace \gamma \geq 0 : \lim\limits_{\| \xi \| \rightarrow 0} \frac{\inf_{x \in \mathbb{R}^d} \inf_{\| \eta \| \geq \| \xi \|} \mathrm{Re} \ q(x,\eta)}{\| \xi \|^{\gamma}}=\infty \right\rbrace. 
\end{align*}
\end{definition}


It is easily derived that
\[
  0 \leq \beta_0 \leq \underline{\beta}_0 \leq \delta_0 \leq 2  \quad \text{ and } \quad 0 \leq  \beta_0 \leq \overline{\delta}_0 \leq \delta_0\leq 2.
\]  

\begin{example}
\begin{itemize}
\item[(a)] For the stable-like process we obtain
\[
  \beta_0 = \underline{\beta}_0 =\inf_{x \in \RR^d} \alpha(x) \quad \text{ and } \quad \delta_0  = \overline{\delta}_0 = \sup_{x \in \RR^d} \alpha(x).
\]
\item[(b)] For the $\alpha$-stable L\'evy process all indices at infinity and zero are $\alpha$. 
\end{itemize}
\end{example}


\begin{theorem} \label{Satz:globWachstumFellerProz}
Let $X$ be a Feller process with symbol $q(x,\xi)$, having bounded coefficients  \\
The we obtain $\mathbb{P}^x$-a.s.
\begin{align}
\lim\limits_{t \rightarrow \infty} \frac{\sup_{s \leq t}\|X_s-x\|}{t^{1/ \gamma}}&=0 \quad	\text{ for all } \gamma < \beta_0 ,
\label{Formel:SatzPfadEigenschaftunendlich1.Absch}\\
\varliminf\limits_{t \rightarrow \infty} \frac{\sup_{s \leq t}\|X_s-x\|}{t^{1/ \gamma}}&=0 \quad	\text{ for all } \gamma < \underline{\beta}_0
\label{Formel:SatzPfadEigenschaftunendlich2.Absch}\\
\intertext{ and if the sector condition \eqref{sector} is satisfied,}
\varlimsup\limits_{t \rightarrow \infty} \frac{\sup_{s \leq t}\|X_s-x\|}{t^{1/ \gamma}}&=\infty \quad	\text{ for all } \gamma > \overline{\delta}_0 ,
\label{Formel:SatzPfadEigenschaftunendlich3.Absch}\\
\lim\limits_{t \rightarrow \infty} \frac{\sup_{s \leq t}\|X_s-x\|}{t^{1/ \gamma}}&=\infty \quad	\text{ for all } \gamma > \delta_0 .
\label{Formel:SatzPfadEigenschaftunendlich4.Absch}
\end{align} 
\end{theorem}

\begin{proposition}
For the OSL process we derive
\[
  \beta_0= \underline{\beta}_0= \inf\limits_{x \in \RR^d} \frac{1}{\Lambda(x)}.
\]
\end{proposition}

\begin{proof}
By Theorem \ref{Satz:obereAbschSymbolReell} there is a constant $C>0$ s.t.
\[
  |q(x, \eta) |  \leq C \| \eta \|^{1/ \Lambda(x)} \quad \text{ for } x \in \RR^d \text{ and } \| \eta \| \leq 1.
\]  
Hence for small values of $\xi $ (i.e. $ \| \xi \| \leq 1$) 
\begin{align*}
\lim\limits_{\| \xi \| \rightarrow 0} \frac{\sup_{x \in \mathbb{R}^d} \sup_{\| \eta \| \leq \| \xi \|} | q(x,\eta)|}{\| \xi \|^{\gamma}} &\leq \lim\limits_{\| \xi \| \rightarrow 0} \frac{\sup_{x \in \mathbb{R}^d} \| \xi \|^{1/\Lambda(x)}}{\| \xi \|^{\gamma}} \\
& = \lim\limits_{\| \xi \| \rightarrow 0} \frac{ \| \xi \|^{\inf_{x \in \mathbb{R}^d} 1/\Lambda(x)}}{\| \xi \|^{\gamma}}  =0,
\end{align*}
if $ \gamma < \inf_{x \in \mathbb{R}^d} 1/\Lambda(x)$, and hence the desires result. Analogoulsy for $\underline{\beta}_0$.
\end{proof}


\begin{corollary} \label{cor:longterm}
Let $X$ be an OSL process satisfying (E4).
It holds $\PP^x$-a.s.
\begin{align*}
  \lim\limits_{t \rightarrow \infty} \frac{\sup_{s \leq t}\|X_s-x\|}{t^{1/ \gamma}}&=0 \quad	\text{ for all } \gamma < \inf\limits_{x \in \RR^d} \frac{1}{\Lambda(x)},\\
  \lim\limits_{t \rightarrow \infty} \frac{\sup_{s \leq t}\|X_s-x\|}{t^{1/ \gamma}}&=\infty \quad	\text{ for all } \gamma > \frac{1}{a} .
\end{align*}
\end{corollary}

\begin{proof}
The first claim directly follows from \ref{Satz:globWachstumFellerProz} and the representation of the indices of the OSL process. 
In order to show the second property we follow the same lines as the proof of Theorem 5.16 in \cite{LevyLike}.
For the maximal process we write a usual
\[  
  \left( X_{\cdot} - x \right)_t^{\ast} := \sup_{s \leq t}\|X_s-x\|.
\]  
We assume  $ \gamma >   \epsilon > \frac{1}{a}$. By Corollary  \ref{Koro:MaximalabschReellesSymbolUntere} it follows for all $ t \geq 1$
\[
 \PP^{x} \left( \left( X_{\cdot} - x \right)_t^{\ast} \leq t^{1/\epsilon} \right) \leq \frac{C}{t} t^{\frac{1}{\epsilon a}}= C t^{\frac{1}{\epsilon a} -1 }
 \]
with a constant $C>0$. We chose $t=t_k=2^{k}, k \in \NN$, and sum up the respective probabilities. We obtain: 
\[
 \sum\limits_{k=1}^{\infty} \PP^{x} \left( \left( X_{\cdot} - x \right)_{t_k}^{\ast} \leq t^{1/\epsilon} \right)  \leq C \sum\limits_{k=1}^{\infty} 2^{k \left(\frac{1}{\epsilon a} -1 \right) } < \infty,
\] 
since $\epsilon > \frac{1}{a}$. By the Borel-Cantelli lemma: 
\[
   \PP^x \left( \varlimsup\limits_{k \rightarrow \infty} \left\lbrace \left( X_{\cdot} - x \right)_{t_k}^{\ast} > t_k^{1/ \epsilon} \right\rbrace \right)=0
 \]
respectively
\[
   \left( X_{\cdot} - x \right)_{t_k}^{\ast} > t_k^{1/ \epsilon} \quad \text{ for almost all } k \in \NN.
\] 
We chose $t \in [t_{k}, t_{k+1}]$ an for sufficiently big $ k$ we obtain for all $\omega$:
\[
   \left( X_{\cdot}(\omega) - x \right)_t^{\ast} \geq \left( X_{\cdot}(\omega) - x \right)_{t_{k}}^{\ast} > t_{k}^{1 / \epsilon } \geq 2^{1 / \epsilon} t^{1 / \epsilon }.
\]
Hence for all $\gamma >   \epsilon $
\[
   t^{-1/\gamma} \left( X_{\cdot} - x \right)_t^{\ast} > 2^{1 / \epsilon} t^{1 / \epsilon -1/ \gamma} \rightarrow \infty \quad \PP^{x}\text{-f.s. for } t \rightarrow \infty
 \]
and the result follows since we have chosen $\epsilon> \frac{1}{a}$. 
\end{proof}



\subsection{$p$-Variation}

Fine properties of the the paths of the process can be derived in the general framework. The additional hypothesis (E4) is not needed. 

\begin{definition}
For $p \in (0, \infty) $ and a c\`adl\`ag function $f : \RR_+ \rightarrow \RR^d$ 
\begin{equation}
  V_p(f,t):=V_p(f, [0,t]):= \sup\limits_{\pi_n} \sum_{i =1}^n \| f(t_j)- f(t_{j-1}) \|^p
\end{equation}
is called \emph{(strong) $p$-variation} of $f$ on $[0,t]$, where the supremum is taken over all partitions \\
$\pi_n= \left(0=t_0<t_1< \hdots < t_n=t \right), n \in \NN,$.

We say that $f$ is of  \emph{finite $p$-variation}, if $V_p(f, t) < \infty$ for all $t \geq 0$.
\end{definition}


A stochastic process $X$ is called \emph{ of finite $p$-variation}, if almost all path are of finite $p$-variation. 

We make use of the following general result on $p$-variation of strong Markov processes (cf. \cite{mansta2005}).

\begin{theorem}\label{Satz:StarkerMarkovProzVariation}
Let $X=(X_t)_{t \geq 0}$ be a strong Markov process. If there exist constants $\alpha >0 , \beta > (3-e)/(e-1) \approx 0,16395$ and 
$C,R_0 >0$, such that
\[
    a(t,R):= \sup\limits_{x \in \RR^d} \sup\limits_{s \leq t} \PP^x( \|X_s -x\| \geq R) \leq C t^{\beta}R^{-\alpha} \quad \text{ for all } t >0 \text{ and } R \in [0,R_0).
\]
Then it holds
\[
   \PP^x(V_p(X_t, t) < \infty) = 1 \quad \text{ for all } p > \frac{\alpha}{\beta}   \text{ and } x \in \RR^d.
\]
\end{theorem}

\begin{proposition}
Let $X$ be an OSL process. Then it holds: 
\[
 \PP^x(V_p(X_t, t) < \infty) = 1 \quad \text{ for all } p > \sup\limits_{x \in \RR^d} \frac{1}{\lambda(x)} \text{ and } x \in \RR^d.
\] 
\end{proposition}

\begin{proof}
For all $t,R>0$
\[
a(t,R)= \sup\limits_{x \in \RR^d} \sup\limits_{s \leq t} \PP^x( \|X_s -x\| \geq R) \leq \sup\limits_{x \in \RR^d} \PP^x \left( \sup\limits_{s \leq t} \|X_s -x\| \geq R \right).
\]
By Theorem \ref{Satz:FellerProzessErstaustrittszeit} there exists a constant $C>0$, such that for all $ R >0$, $x \in \RR^d$ and $t >0$:
\[
   \PP^x \left( \sup\limits_{s \leq t} \|X_s -x\| \geq R \right)\leq Ct \sup_{\| y- x \| \leq R} \sup_{\| \xi \| \leq 1/R} |q(y, \xi)|.
\]   
For all $p> \sup_{x \in \RR^d} \beta^{x}_{\infty} =  \sup_{x \in \RR^d} \frac{1}{\lambda(x)}$ there exists due to the representation of $\beta^{x}_{\infty} $ off an OSL process a constant $R_0>0$ with
\[
a(t,R) \leq  Ct  \sup\limits_{x \in \RR^d} \sup_{\| y- x \| \leq R} \sup_{\| \xi \| \leq 1/R} |q(y, \xi)| \leq \tilde{C} t R^{-p} \quad \text{ for all } t>0 \text{ and } R \in [0,R_0).
\]
for $ \alpha=\sup_{x \in \RR^d} \frac{1}{\lambda(x)}$, $ \beta=1$ and $R_0>0$ . The claim follows by Theorem \ref{Satz:StarkerMarkovProzVariation}.
\end{proof}

\textbf{Acknowledgements:} Financial support by the DFG (German Science Foundation) for the project SCHN 1231/2-1 of AS is gratefully acknowledged.





\bibliographystyle{alpha}

\bibliography{OperatorStableLikeRevision01an.bib}

\begin{thebibliography}{BSW13}

\bibitem[App09]{Apple}
D.~Applebaum.
\newblock Lévy processes and stochastic calculus.
\newblock Cambridge Studies in Advanced Mathematics, Cambridge, 2009.

\bibitem[Bas]{BassStableLike}
R.~F. Bass.
\newblock Uniqueness in law for pure jump {M}arkov processes.
\newblock {\em Probability Theory and Related Fields}, \textbf{79(2)} (1988):\
  271--287.

\bibitem[BF75]{KonstLevyMass}
C.~Berg and G.~Forst.
\newblock {\em Potential Theory on Locally Compact Abelian Groups}.
\newblock Springer, 1975.

\bibitem[BG68]{blumenthalget}
R.~M. Blumenthal and R.~K. Getoor.
\newblock {\em Markov {P}rocesses and {P}otential {T}heory}.
\newblock Academic Press, New York 1968.

\bibitem[BMBS]{Water}
D.~Benson, M.~M. Meerschaert, B.~B\"aumer, and H.~P. Scheffler.
\newblock Aquifer operator-scaling and the effect on solute mixing and
  dispersion.
\newblock {\em Water Resour. Res.}, \textbf{42} (2006):\ 1--18.

\bibitem[BMS]{Bierme}
H.~Biermé, M.~M. Meerschaert, and H.-P. Scheffler.
\newblock Operator scaling stable random fields.
\newblock {\em Stochastic Processes and their Applications}, \textbf{117(3)}
  (2007):\ 312--332.

\bibitem[BSW13]{LevyLike}
B.~Böttcher, R.~L. Schilling, and J.~Wang.
\newblock Lévy matters iii: Lévy-type processes: Construction, approximation
  and sample path properties.
\newblock Springer, Berlin, 2013.

\bibitem[CJ81]{cinlarjacod}
E.~Cinlar and J.~Jacod.
\newblock {\em Progress in Probability and Statistics, vol 1}, chapter
  Representation of Semimartingale Markov Processes in Terms of Wiener
  Processes and Poisson Random Measures, pages \ 159--242.
\newblock Birkhäuser, Boston, 1981.

\bibitem[CJPS]{vierleute}
E.~Cinlar, J.~Jacod, P.~Protter, and M.~J. Sharpe.
\newblock Semimartingales and {M}arkov {P}rocesses.
\newblock {\em Z. Wahrscheinlichkeitstheorie verw. Gebiete}, \textbf{54}
  (1980):\ 161--219.

\bibitem[Kü]{MomenteFranzKuehn}
F.~Kühn.
\newblock Existence and estimates of moments for {L}évy-type processes.
\newblock {\em Stochastic Processes and their Applications}, \textbf{127(3)}
  (2017):\ 1018--1041.

\bibitem[Man]{mansta2005}
M.~Manstavi\v{c}ius.
\newblock A non-markovian process with unbounded $p$-variation.
\newblock {\em Electron. Commun. Probab.}, \textbf{10} (2005):\ 17--28.

\bibitem[MS01]{Peter}
M.~M. Meerschaert and H.-P. Scheffler.
\newblock Limit distributions for sums of independent random vectors.
\newblock Wiley Series in Probability and Statistics, 2001.

\bibitem[Nol03]{Nolan}
John~P. Nolan.
\newblock Modeling financial data with stable distributions.
\newblock In Svetlozar~T. Rachev, editor, {\em Handbook of Heavy Tailed
  Distributions in Finance}, volume~1 of {\em Handbooks in Finance}, pages \
  105--130. North-Holland, Amsterdam, 2003.

\bibitem[Sat99]{Sato}
K.~Sato.
\newblock Lévy processes and infinitely divisible distributions.
\newblock Cambridge Studies in Advanced Mathematics, Cambridge, 1999.

\bibitem[Scha]{Pfad}
R.~L. Schilling.
\newblock Growth and {H}ölder conditions for the sample paths of {F}eller
  processes.
\newblock {\em Probab. Theory Rel. Fields}, \textbf{112} (1998):\ 565--611.

\bibitem[Schb]{Schnurr}
A.~Schnurr.
\newblock Generalization of the {B}lumenthal-{G}etoor index to the class of
  homogeneous diffusions with jumps and some applications.
\newblock {\em Bernoulli}, \textbf{19(5A)} (2013):\ 2010--2032.

\bibitem[Sch09]{Alex}
A.~Schnurr.
\newblock {\em The symbol of a Markov semimartingale}.
\newblock PhD thesis, TU Dresden, 2009.

\bibitem[Sch18]{Daniel}
D.~Schulte.
\newblock {\em \"Uber operator-stable-like Prozesse und ihre Eigenschaften}.
\newblock PhD thesis, Universität Siegen, 2018.

\bibitem[Sha]{Sharpe}
M.~Sharpe.
\newblock Operator stable probability distributions on vector groups.
\newblock {\em Trans. Amer. Math. Soc.}, \textbf{136} (1969):\ 51--65.

\bibitem[VL]{VanLoan}
C.~Van~Loan.
\newblock The sensitivity of the matrix exponential.
\newblock {\em SIAM Journal on Numerical Analysis}, \textbf{6(14)} (1977):\
  971--981.

\bibitem[{Yan}]{ChineseSDE}
X.~{Yang}.
\newblock {Hausdorff dimension of the range and the graph of stable-like
  processes}.
\newblock {\em J. Theor. Prob.}, \textbf{31} (2018):\ 2412--2431.

\end{thebibliography}


\pagebreak
\appendix

\section{Appendix: Matrix exponential, generalized polar coordinates and estimates}

As described in the Introduction, the operator norm on $L(\RR^d)$ is written as
\begin{equation}
\|A \|:= \sup\limits_{\|x\| =1} \|Ax \| = \sup\limits_{\|x\| \leq 1} \|Ax \| .
\end{equation} 
and the matrix exponential is defined by
\begin{equation}
r^A:= \exp( A \ln r)=\sum\limits_{k=0}^{\infty} \frac{A^k}{k!}(\ln r)^k \quad \in L(\RR^d),
\end{equation}
furthermore in our considerations for every $x\in\bbr^d$,  $E(x) \in L(\RR^d)$ and
$$ \lambda(x): = \min \lbrace \mathrm{Re } \ \sigma(x) : \sigma(x) \text{ is eigenvalue of } E(x) \rbrace$$
and
$$ \Lambda(x): =\max \lbrace \mathrm{Re } \ \sigma(x) : \sigma(x) \text{ is eigenvalue of } E(x)  \rbrace.$$

\begin{theorem} \label{Satz:AllgAbschDiagOperatorEW}
Let $ E : \RR^d \rightarrow L(\RR^d)$ be symmetric. Then, there exists a constant $C>0$, such that
\begin{equation}
\| r^{E(x)} \| \leq 
\begin{cases}
  C r^{\lambda(x)} & \text{ for } 0 < r < 1 ,\\
  C r^{\Lambda(x)} &\text{ for } r \geq 1.
\end{cases}
\end{equation}
\end{theorem}

\begin{proof}
If $E(x)$ is symmetric, there exist orthogonal matices $ \RR^d \ni x \mapsto O(x) \in GL(\RR^d)$ and diagonal matrices $ \RR^d \ni x \mapsto D(x) \in GL(\RR^d)$, such that
$$E(x)= O(x)D(x)O(x)^{-1}.$$
For $r>0$ we get by well know facts on the matrix exponential
$$r^{E(x)}=O(x)r^{D(x)}O(x)^{-1}.$$
Hence by the sub-multiplicativity of the operator norm
$$ \| r^{E(x)} \| \leq \| O(x) \|  \| r^{D(x)}\| \| O(x)^{-1} \| = \| r^{D(x)}\| .$$
For diagonal matrices $D(x):=\text{diag} (a_1(x), \ldots , a_d(x))$
$$r^{D(x)}=\text{diag} (r^{a_1(x)}, \ldots , r^{a_d(x)}).$$
with the taxi norm $| \cdot |$ we get for $0<r < 1$:
$$ | r^{D(x)} |= \sum_{k=1}^{d}r^{a_k(x)}  \leq d  r^{\lambda_{D(x)}(x)}.$$
Since $D(x)$ and $E(x)$ have the same eigenvalues, the claim follows in this case. The second case is derived by  $\max\{a_1(x), \ldots , a_d(x)\}= \Lambda_{D(x)}(x)$ and $r \geq 1$ in an analogous way. In $\RR^d$ resp. $L(\RR^d)$ all norms are equivalent. Hence the estimate holds for operator norm with a different constant $C>0$. 
\end{proof}
Next we give a short overview on so called generalized polar coordinates (cf. \cite{Bierme} Section 2):
Let $E(x) \in L(\RR^d)$ with $\lambda(x) >0$ for all $x \in \RR^d$.
By Lemma 6.1.5 in \cite{Peter} 
\begin{equation}
\| \xi \|_0 := \int_0^1 \| r^{E(x)} \xi \| \frac{dr}{r}, \quad x,\xi \in \RR^d,
\label{Formel:NormVerallPolarkordinaten}
\end{equation}
defines a norm $ \| \cdot \| $ on $ \RR^d$ such that $ \Psi : (0,\infty)  \times S_0 \rightarrow \Gamma, \Psi(r, \theta)= r^{E(x)} \theta$ is a homeomorphism where $S_0:= \lbrace \xi \in \RR^d: \| \xi\|_0 =1 \rbrace$.
Since for every $\xi \in \Gamma$ and every $x \in \RR^d$ the function $ r \mapsto \| r^{E(x)} \xi \|$ is monotonically increasing, every $\xi \in \Gamma$ can be represented uniquely by
\[
     \xi= \tau_x(\xi)^{E(x)} l_x(\xi)
\]     
where $\tau_x(\xi)>0$  is called \emph{radial component} and $l_x(\xi) \in S_0$ \emph{direction}. The pair 
\begin{equation}
  \left( \tau_x(\xi), l_x(\xi) \right)
\end{equation} 
is called \emph{generalized polar coordinates} of $\xi$ w.r.t. the matrix $E(x)$. The functions $\tau_x$ and $l_x$ are continuous. They have the properies
\begin{itemize}
\item[(i)] $\tau_x(\xi) \rightarrow \infty$ for $ \| \xi \| \rightarrow \infty$ and $\tau_x(\xi) \rightarrow 0$ for $ \| \xi  \| \rightarrow 0$;
\item[(ii)] $\tau_x(-\xi)= \tau_x(\xi)$ and $l_x(-\xi)=-l_x(\xi)$;
\item[(iii)] $\tau_x(r^{E(x)} \xi)= r \tau_x(\xi)$ and $l_x(r^{E(x)}\xi) =l_x(\xi)$ for all $r>0$.
\end{itemize}
In addition $S_0= \lbrace \xi \in \RR^d : \tau_x(\xi)=1 \rbrace$ is a compact set. 

\begin{remark}
If $E(x)$ is symmetric, one can chose  $\| \cdot \|_0 = \| \cdot \| $ s.t. $S_0= \SD$.
\end{remark}

The following lemma contains bounds for the growth rate of  $\tau_x(\xi)$ which depend on the real part of the smallest eigenvalue of $E(x)$. 

\begin{lemma}\label{Lemma:AbschVerallgPolarkoordianten}
Let $E : \RR^d \rightarrow L(\RR^d)$ be symmetric. Then there exist constants $C_1, \hdots, C_4>0$, such that
\begin{itemize}
\item[(i)] for all  $\|\xi \| \leq 1$ or $\tau_x(\xi) \leq 1$ it holds
\[
  C_1 \| \xi \|^{1/\lambda(x) } \leq \tau_x(\xi) \leq C_2 \| \xi \|^{1/\Lambda(x)};
\]
\item[(ii)] for all $\| \xi \| \geq 1$ oder $\tau_x(\xi) \geq 1$ it holds:
\[
  C_3 \| \xi \|^{1/\Lambda(x)} \leq \tau_x(\xi) \leq C_4 \| \xi \|^{1/\lambda(x)}.
\]
\end{itemize}
\end{lemma}

\begin{proof}
We only show the first two inequalities, the other to being proved analogously. By the Cauchy-Schwarz inequality an Theorem
\ref{Satz:AllgAbschDiagOperatorEW} we obtain
\[
   \| \xi \| = \| \tau_x(\xi)^{E(x)} l_x (\xi) \| \leq \| \tau_x( \xi)^{E(x)} \| \cdot \| l_x(\xi) \| \leq C \tau_x(\xi)^{\lambda(x)}
\]
for all  $\tau_x(\xi) \leq 1 $ and a constant $C>0$. This yields
\[
 \tau_x(\xi) \geq C_1 \| \xi \|^{1/ \lambda(x)}
\] 
for $\tau_x( \xi) \leq 1$. Here,  $\tau_x( \xi) \leq 1$ is by definition of the norm $\| \cdot \|_0= \| \cdot \|$ (cf. Formula \eqref{Formel:NormVerallPolarkordinaten}) equivalent to $\| \xi \| \leq 1$.
\\
For the upper bound, Theorem \ref{Satz:AllgAbschDiagOperatorEW} yields:
\[
   \| \tau_x(\xi) ^{-E(x)} \| \leq C \tau_x(\xi)^{-\Lambda(x)} 
\] 
with a constant $C>0$. Furthermore, it holds $l_x(\xi)= \tau_x(\xi)^{-E(x)} \xi$ and hence it follows
\[
     1= \| l_x(\xi) \| \leq \| \tau_x( \xi)^{-E(x)} \| \cdot \| \xi \| \leq C \tau_x(\xi)^{-\Lambda(x)} \| \xi \|.
\]
Writing this in terms of $ \tau_x(\xi)$ yields the result. 
\end{proof}

In order to establish the subsequent lemma we need the following. (cf. formula ($1.3$) in \cite{VanLoan}).
\begin{lemma}    \label{Lemma:VanLoanFormel}
For $A,B \in GL(\RR^d)$ and $t \geq 0$ the following identity holds
\begin{equation}
e^{(A+B)t}=e^{At}+ \int_0^t e^{A(t-s)}Be^{(A+B)s} \ ds.
\end{equation} 
\end{lemma}

\begin{lemma}   \label{Lem:AbschaetzungDiffExponenten}
Let $E(x)$ be an admissible exponent. For $\delta >0$ there exists constants $C,\tilde{C}>0$, such that for all $ r \in (0,1) $ and all $x,y \in \RR^d$:
\begin{equation} \label{AbschaetzDiffExp}
\| r^{E(x)}-r^{E(y)} \| \leq C \| E(x)-E(y)\| \cdot r^{a-\delta} \leq \tilde{C} \|x-y\| \cdot r^{a-\delta}.
\end{equation}
\end{lemma}

\begin{proof}
The second inequality directly follows from the Lipschitz property of the exponent. It remains to show the first one. 
We use Lemma \ref{Lemma:VanLoanFormel} with $A:=-E(y), B:=E(y)-E(x)$ and $t:=-\ln(r)> 0$ for $r \in (0,1)$. Hence, 
\begin{align*}
  r^{E(x)}-r^{E(y)} &= e^{(A+B)t}-e^{At} \\
    &= \int_0^t e^{A(t-s)}Be^{(A+B)s} \ ds \\
    &= \int_0^{-\ln (r)} e^{(-E(y))(- \ln (r) -s )}(E(y)-E(x)) e^{-(E(x))s}  ds \\
    &= \int_0^{-\ln (r)} e^{E(y)(\ln (r) +s)}(E(y)-E(x))e^{-E(x)s} \ ds.
\end{align*}
This yields that
\[
   \| r^{E(y)} - r^{E(x)} \| \leq \| E(y)-E(x) \| \int_0^{- \ln (r)} \|e^{E(y)(\ln (r) +s) } \| \cdot \| e^{-E(x)s}\| \ ds.
\]   
By a change of variables $s=-\ln (v)$, this is equal to
\[
 \int_r^1 \| e^{E(y)(\ln (r) - \ln (v))} \| \cdot \| e^{E(x)\ln (v)}\| \frac{1}{v} \ dv = \int_r^1 \left\Vert \left(\frac{r}{v}\right)^{E(y)}\right\Vert \cdot \| v^{E(x)}\| \frac{1}{v} \ dv.
\]
Since $v \leq 1$ and $r/v \leq 1$, the integral can be estimated by using Proposition \ref{Prop:AbschMEExponent} with constants $C_1,C_2>0$.
\begin{align*}
C_1C_2\int_r^1 \left(\frac{r}{v}\right)^{a} v^{a} \frac{1}{v} \ dv &= C_1C_2 r^{a} \int_r^1 \frac{1}{v} \ dv \\
&=C_1C_2 r^{a} (- \ln (r)) \\
& \leq C  r^{a-\delta}
\end{align*}
Here,  $r^{\delta}(-\ln (r))$ is bounded, since $\delta >0$ and $r \in (0,1)$. 
\end{proof}

\begin{remark}
In Sections 2 and 3 we tacitly assume that the constant $\delta>0$ in Lemma \ref{Lem:AbschaetzungDiffExponenten} is chosen in such a way that $ a- \delta > \frac{1}{2}$ is satisfied.
\end{remark}

\end{document}